\documentclass[a4paper,reqno]{amsart}

\usepackage{amsfonts}
\usepackage{amsmath,amssymb,amsthm,amsthm}
\usepackage{mathtools}
\usepackage{etoolbox}
\usepackage{mathrsfs}
\usepackage[english]{babel}
\usepackage[utf8]{inputenc}
\usepackage{hyperref}
\usepackage{microtype}
\usepackage{fullpage}
\usepackage[leftcaption]{sidecap}
\sidecaptionvpos{figure}{m}

\DisableLigatures{encoding = *, family = * }

\newtheorem{theorem}{Theorem}[section]
\newtheorem{definition}[theorem]{Definition}
\newtheorem{lemma}[theorem]{Lemma}
\newtheorem{proposition}[theorem]{Proposition}
\newtheorem{remark}[theorem]{Remark}
\newtheorem{problem}{Problem}[section]

\numberwithin{equation}{section}

\DeclareMathOperator*{\argmin}{arg\,min}
\usepackage{soul}
\usepackage{xcolor}

\newcommand{\mS}{\mathcal{S}}
\newcommand{\aT}{{\bf a}_T}
\newcommand{\bT}{{\bf b}_T}

\begin{document}
	
\title[Multilevel SHM via Optimal Control]{Multilevel Selective Harmonic Modulation via Optimal Control}
\author{Umberto Biccari\textsuperscript{\,$\ast$\,$\dagger$}}  
\author{Carlos Esteve-Yag\"ue\textsuperscript{\,$\ast$}}
\author{Deyviss Jes\'us Oroya-Villalta\textsuperscript{\,$\dagger$}}

\address{\textsuperscript{$\ast$}\, Chair of Computational Mathematics, Fundaci\'on Deusto, Avenida de las Universidades 24, 48007 Bilbao, Basque Country, Spain} 
\address{\textsuperscript{$\dagger$}\, Facultad de Ingenier\'ia, Universidad de Deusto, Avenida de las Universidades 24, 48007 Bilbao, Basque Country, Spain.} 

\email{umberto.biccari@deusto.es, u.biccari@gmail.com, carlos.esteve@deusto.es, djoroya@deusto.es}

\thanks{This project has received funding from the European Research Council (ERC) under the European Union’s Horizon 2020 research and innovation programme (grant agreement NO: 694126-DyCon). The work of U.B. is partially supported by the Elkartek grant KK-2020/00091 CONVADP of the Basque government, by the Air Force Office of Scientific Research (AFOSR) under Award NO: FA9550-18-1-0242 and by the Grant PID2020-112617GB-C22 KILEARN of MINECO (Spain)}

\begin{abstract}
We consider the \textit{Selective Harmonic Modulation} (SHM) problem, consisting in the design of a staircase control signal with some prescribed
frequency components. In this work, we propose a novel methodology to address SHM as an optimal control problem in which the admissible controls are piecewise constant functions, taking values only in a given finite set. In order to fulfill this constraint, we introduce a cost functional with piecewise affine penalization for the control, which, by means of Pontryagin’s maximum principle, makes the optimal control have the desired staircase form. Moreover, the addition of the penalization term for the control provides uniqueness and continuity of the solution with respect to the target frequencies. Another advantage of our approach is that the number of switching angles and the waveform need not be determined a priori. Indeed, the solution to the optimal control problem is the entire control signal, and therefore, it determines the waveform and the location of the switches. We also provide numerical examples in which the SHM problem is solved by means of our approach.
\end{abstract}


\maketitle

\section{Introduction and motivations}\label{Section1}

Selective Harmonic Modulation (SHM) \cite{Sun1992,Sun1996} is a well-known methodology in power electronics engineering, employed to improve the performance of a converter by controlling the phase and amplitude of the harmonics in its output voltage. As a matter of fact, this technique allows to increase the power of the converter and, at the same time, to reduce its losses. In broad terms, SHM consists in generating a \textit{control signal} with a desired harmonic spectrum by modulating some specific lower-order Fourier coefficients. In practice, the signal is constructed as a step function with a finite number of switches, taking values only in a given finite set. Such a signal can be fully characterized by two features (see Fig. \ref{fig:exampleSHE}): 
\begin{itemize}
	\item[1.] The \textit{waveform}, i.e. the sequence of values that the function takes in its domain.
	\item[2.] The \textit{switching angles}, i.e. the sequence of points where the signal switches from one value to following one. 
\end{itemize}

Using this simple characterization of the signal, in many practical situations, the SHM problem is reduced to a finite-dimensional optimization one in which, for a given suitable waveform, the aim is to find the optimal location of the switching angles. However, this approach has the  difficulty of choosing a suitable waveform,  which may be quite cumbersome in some situations.  In fact, even determining the number of switching angles is not straightforward in general. To overcome these difficulties,  we propose a new approach to SHM based on control theory: the Fourier coefficients of the signal are identified with the terminal state of a controlled dynamical system, where the control is actually the signal, solution to the SHM problem. We then look for piecewise constant controls, taking values only in a given finite set, and satisfying the prescribed terminal condition (see Section \ref{sec:Contributions} for more details). 

One of the main difficulties in our approach is  that the constraints on the control, which must have staircase form (taking values only in a given finite set), prevent us from implementing the standard numerical tools in optimal control. Specifically, one of the most popular methodologies to solve optimal control problems is the combination of automatic differentiation with nonlinear convex optimization, achieving a good algorithmic performance. However, the use of these optimizers is restricted to cases where the space of admissible controls is convex,  not being directly applicable to our problem. In order to bypass this obstruction, we consider a variant of the optimal control problem, removing the staircase constraint on the control, and adding a suitable convex penalization term, which makes the solution have the desired staircase form.

The main contributions of the present paper are the following ones:
\begin{enumerate}
	\item[1.] We reformulate the SHM problem as an optimal control one, with a staircase-form constraint on the control. An advantage of this formulation is that neither the waveform of the solution nor the number of switching angles need to be a priori determined.  
	\item[2.] We introduce a penalization term for the control which implicitly induces the desired staircase property on the solution to the optimal control problem. Different choices of the penalization term can give rise to solutions with different waveform.
	\item[3.] For this penalization term, we prove uniqueness and continuity of the solution with respect to the target frequencies. We point out that this continuity is highly desirable in real applications of SHM, and sometimes, difficult to achieve.
	\item[4.] We also provide numerical examples, where we solve the SHM problem through our approach. These examples confirm that the solution obtained via our methodology is, effectively, continuous with respect to the target frequencies.
\end{enumerate}

Let us mention that optimal discrete-valued control problems as the one presented in this work has already been discussed in \cite{lee1999control} (see also \cite{wu2009filled,yu2013optimal}) using a different approach, the so-called \textit{control parametrization enhancing technique}. However, this method presents several drawbacks which in our case do not arise. First of all, the technique presented in \cite{lee1999control} requires to pre-fix the number of switching angles, which in many cases is not straightforward. Secondly, it does not ensure the staircase-form of the optimal controls, which is instead provided by our methodology. Finally, our approach also yields uniqueness and continuity of the optimal control with respect to the target frequencies, which is not discussed in \cite{lee1999control} for the control parametrization enhancing technique.

This document is structured as follows. In Section \ref{sec:math_formulation}, we introduce the mathematical formulation of the general SHM problem. In Section \ref{sec:SHE_finite-dim_pbm}, we recall the classical methodology casting the SHM problem through finite-dimensional optimization and we show the main criticalities related to this approach. In Section \ref{sec:Contributions}, we present the new approach to SHM  as an optimal control problem, and state our main results concerning the uniqueness and stability of the solution. In Section \ref{sec:Proof}, we give the proofs of the theoretical results presented in Section \ref{sec:Contributions}. Section \ref{sec:Simulations} is devoted to some numerical examples of concrete SHM problems that we have solved by means of our methodology. Finally, in Section \ref{sec:conclusions}, we summarize and comment the conclusions of our work.

\section{Mathematical formulation of the SHM problem}\label{sec:math_formulation}

This section is devoted to the mathematical formulation of the SHM problem and to introduce the notation that will be used throughout the paper. Let 
\begin{align}\label{eq:Udef}
	\mathcal{U} = \{u_1, \ldots, u_L\}
\end{align}
be a given set of $L\geq 2$ real numbers satisfying
\begin{align*}
	u_1 = -1, \; u_L = 1 \;\text{ and } \; u_k<u_{k+1}, \quad\; \forall k\in \{1,\ldots, L\}.
\end{align*}

The goal is to construct a step function $u(t):[0,2\pi)\to\mathcal U$, with a finite number of switches, such that some of its lower-order Fourier coefficients take specific values prescribed a priori.

Due to applications in power converters,  it is typical to only consider functions with \textit{half-wave symmetry}, i.e. 
\begin{align}\label{eq:half-wave symmetry}
	u(t + \pi) = -u(t)\quad \forall t \in [0,\pi).
\end{align}

In view of \eqref{eq:half-wave symmetry}, in what follows, we will only work with the restriction $u\vert_{[0,\pi)}$, which, with some abuse of notation, we still denote by $u$. Moreover, as a consequence of this symmetry, the Fourier series of $u$ only involves the odd terms (as the even terms just vanish), i.e.
\begin{align*}
	u(t) = \sum_{\underset{j\, odd}{j \in \mathbb{N}}} a_j \cos(jt)+ \sum_{\underset{k\, odd}{j \in \mathbb{N}}}  b_j \sin(jt),
\end{align*}
with
\begin{equation}\label{eq:an}
	a_j = \frac{2}{\pi} \int_0^\pi u(\tau ) \cos(j \tau)\,d\tau, \quad\quad\quad b_j = \frac{2}{\pi} \int_0^\pi u(\tau)  \sin(j \tau)\,d\tau.
\end{equation}

As we anticipated, we are only considering piecewise constant functions with a finite number of switches, taking values only in $\mathcal{U}$. In other words, we look for functions $u: [0,\pi)\to \mathcal{U}$ of the form 
\begin{align}\label{eq:uExpl}
	u (t)= \sum_{m=0}^M s_m\chi_{[\phi_m,\phi_{m+1})} (t), \quad M\in\mathbb{N} 
\end{align}
for some $\mS = \{ s_m\}_{m=0}^M$ satisfying 
\begin{align*}
	s_m\in \mathcal{U}\;\text{ and }\; s_m\neq s_{m+1}\;\text{ for all }\;m\in \{0,\ldots, M\}
\end{align*}
and $\Phi = \{ \phi_m\}_{m=1}^{M}$ such that
\begin{align*}
	0= \phi_0 < \phi_1 <\ldots < \phi_M < \phi_{M+1} = \pi.
\end{align*}

In \eqref{eq:uExpl}, $\chi_{[\phi_m,\phi_{m+1})}$ denotes the characteristic function of the interval $[\phi_m,\phi_{m+1})$. With these notations, we can  define the waveform and the switching angles as follows.

\begin{definition}\label{def: waveform and switching angles}
For a function $u: [0,\pi) \to \mathcal{U}$ of the form \eqref{eq:uExpl}, we refer to $\mS$ as the \emph{waveform} and to $\Phi$ as the \emph{switching angles}.
\end{definition}

Observe that any $u$ of the form \eqref{eq:uExpl} is fully characterized by its waveform and switching angles. An example of such a function is given in Fig. \ref{fig:exampleSHE}. 

\begin{SCfigure}[1.3][h]
	\centering
	\includegraphics[scale=0.3]{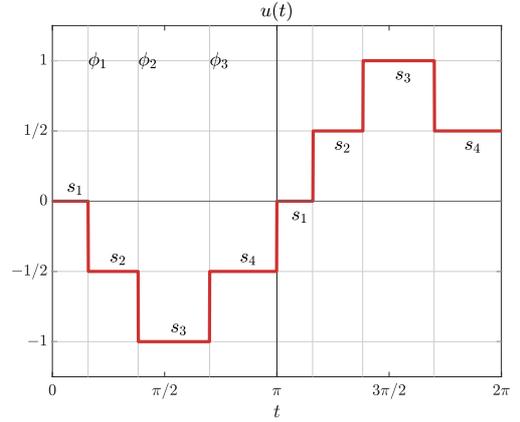} 
	\caption{A possible solution to the SHM Problem, where we considered the control-set $\mathcal{U} = \{-1, -1/2, 0, 1/2, 1\}$. We show the switching angles $\Phi$ and the waveform $\mS$ (see Definition \ref{def: waveform and switching angles}). The function $u(t)$ is displayed on the whole interval $[0,2\pi)$ to highlight the half-wave symmetry defined in \eqref{eq:half-wave symmetry}.}
	\label{fig:exampleSHE}
\end{SCfigure}

In the practical engineering applications that motivated our study, due to technical limitations, it is preferable to employ signals taking consecutive values in $\mathcal{U}$. In the sequel, we will refer to this property of the waveform as the \emph{staircase property}. We can rigorously formulate this property as follows.

\begin{definition}\label{def:staircase prop}
We say that a signal $u$ of the form \eqref{eq:uExpl} fulfills the \emph{staircase property} if its waveform $\mS$ satisfies
\begin{align}\label{eq:staircase prop}
	(s_m^{min},s_{m}^{max}) \cap \mathcal{U} = \emptyset, \quad \text{for all } m\in \{ 0, \ldots, M-1 \},
\end{align}
where $s^{min}_m = s_m\wedge s_{m+1}$ and $s^{max}_m = s_m \vee s_{m+1}.$
\end{definition}

Note that when $\mathcal{U} = \{-1,1\}$ (which is known in the SHM literature as the bi-level problem), this property is satisfied for any $u$ of the form \eqref{eq:uExpl}.

We can now formulate the SHM problem as follows.

\begin{problem}[SHM]\label{pb:SHEp}
Let $\mathcal{U}$ be given as in \eqref{eq:Udef}, and let $\mathcal{E}_a$ and $\mathcal{E}_b$ be finite sets of odd numbers of cardinality $\vert\mathcal{E}_a\vert = N_a$ and $\vert\mathcal{E}_b\vert = N_b$ respectively. For any two given vectors $\aT \in \mathbb{R}^{N_a}$ and $\bT \in \mathbb{R}^{N_b} $, we want to construct a function $u: [0,\pi)\to\mathcal{U}$ of the form \eqref{eq:uExpl}, satisfying \eqref{eq:staircase prop}, such that the vectors ${\bf a} \in \mathbb{R}^{N_a}$ and ${\bf b} \in \mathbb{R}^{N_b}$, defined as
\begin{align}\label{vectors a and b}
	{\bf a} = \big( a_j \big)_{j\in \mathcal{E}_a} \qquad \text{and} \qquad
	{\bf b} = \big( b_j \big)_{j\in \mathcal{E}_b}
\end{align}
satisfy ${\bf a} = \aT$ and ${\bf b} = \bT$, where the coefficients $a_j$ and $b_j$ in \eqref{vectors a and b} are given by \eqref{eq:an}.
\end{problem}  

\begin{remark}[SHE]\label{remark:SHE}
In Problem \ref{pb:SHEp}, we gave a very general formulation of SHM. This formulation contains also the so-called \emph{Selective Harmonic Elimination (SHE)} problem (\cite{Sun1996}), in which the target vectors are such that 
\begin{displaymath}
	\begin{array}{ll}
		(a_T)_1 \neq 0  \hspace{1em} (a_T)_{i\neq1} = 0 & \quad\text{for all } i \in \mathcal{E}_a 
		\\[3pt]
		(b_T)_1 \neq 0  \hspace{1em} (b_T)_{j\neq1} = 0 & \quad\text{for all } j \in \mathcal{E}_b. 
	\end{array} 
\end{displaymath}

SHE is of great relevance in the electric engineering literature. Its objective is to generate a signal with amplitude 
\begin{align*}
	m_1 = \sqrt{a_1^2+b_1^2}
\end{align*}
and phase 
\begin{align*}
	\varphi_1=\arctan\left(\frac{b_1}{a_1}\right),
\end{align*}
removing some specific high-frequency components. In this way, SHE may be understood as a generator of clean Fourier modes through a staircase signal. 
\end{remark}

\section{SHM via finite-dimensional optimization}\label{sec:SHE_finite-dim_pbm}

A typical approach to the SHM Problem \ref{pb:SHEp} (\cite{Konstantinou2010,perez20172n,Yang2015}) is to look for solutions $u$ with a specific waveform $\mathcal{S}$ a priori determined, optimizing only over the location of the switching angles $\Phi$. Note that, for a fixed waveform $\mS$, the Fourier coefficients of a function $u$ of the form \eqref{eq:uExpl} can be written in terms of the switching angles $\Phi$ in the following way:
\begin{align*}
	& a_j = a_j(\Phi) =  \frac{2}{j\pi} \sum_{m=0}^{M} s_m \Big[\sin(j\phi_{m+1}) -\sin(j\phi_{m})\Big]
	\\[5pt]
	& b_j = b_j(\Phi) = \frac{2}{j\pi} \sum_{m=0}^{M} s_m \Big[\cos(j\phi_{m}) -\cos(j\phi_{m+1})\Big]
\end{align*}

Hence, for two sets of odd numbers $\mathcal{E}_a$ and $\mathcal{E}_b$ as in Problem \ref{pb:SHEp}, and any fixed $\mS$, we can define the functions
\begin{equation}\label{eq: a_S b_S}
	{\bf a}_\mS (\Phi) := \big(a_j (\Phi)\big)_{j\in \mathcal{E}_a} \in \mathbb{R}^{N_a}, \quad\quad\quad {\bf b}_\mS (\Phi) \hspace{0.1em}:= \big(b_j (\Phi)\big)_{j\in \mathcal{E}_b} \in \mathbb{R}^{N_b}
\end{equation}
which associate, to any sequence of switching angles $\{\phi_m\}_{m=1}^{M}$, the corresponding Fourier coefficients. Therefore, SHM can be cast as a finite-dimensional optimization problem in the following way.

\begin{problem}[Optimization problem for SHM]\label{pb:SHE opt}
Let $\mathcal{E}_a$, $\mathcal{E}_b$, $\aT$, and $\bT$ be given as in Problem \ref{pb:SHEp}. Let $\mathcal S := \{ s_m\}_{m=0}^M$ be a fixed waveform satisfying \eqref{eq:staircase prop}. We look for a sequence of switching angles $\Phi = \{\phi_m\}_{m=1}^{M}$ solution to the following minimization problem:
\begin{displaymath}
	\begin{array}{l}
		\displaystyle\min_{\Phi \in [0,\pi]^{M}} \bigg( \|{\bf a}_\mS (\Phi) - \aT\|^2 + \| {\bf b}_\mS (\Phi) - \bT\|^2\bigg)
		\\[10pt]
		\mbox{subject to: } 0 = \phi_0 <\phi_1 < \ldots < \phi_{M} < \phi_{M+1} = \pi,
	\end{array} 
\end{displaymath}
where ${\bf a}_\mS (\Phi)$ and ${\bf b}_\mS (\Phi)$ are defined as in \eqref{eq: a_S b_S}.
\end{problem}

At this regard, it is important to notice that the optimization Problem \ref{pb:SHE opt} solves the original SHM Problem \ref{pb:SHEp} only when the minimum equals zero. This makes necessary to fully characterize the space of targets $(\aT,\bT)$ for which the solution of Problem \ref{pb:SHE opt} is a solution of Problem \ref{pb:SHEp}. With this aim, we will define the the \textit{optimal value} and the \textit{solvable set} as follows.

\begin{definition}[optimal value]
We call \emph{optimal value} $V_{\mS}:\mathbb{R}^{N_a}\times \mathbb{R}^{N_b} \rightarrow \mathbb{R}$, the function that takes as input variables the target vectors $\aT$ and $\bT$ and returns the optimal value of the Problem \ref{pb:SHE opt}.
\end{definition}

\begin{definition}[solvable set]
We define a \emph{solvable set} $\mathcal{R}_{\mS}$ as:
\begin{align*}
	\mathcal{R}_{\mS} = \Big\{(\aT,\bT)\in \mathbb{R}^{N_a+N_b}\;:\; V_{\mS}(\aT,\bT) = 0\Big\}
\end{align*}
\end{definition}

Furthermore, we define the following \textit{policy} function which maps the solutions of Problem \ref{pb:SHE opt} into the set $\mathcal{R}_{\mS}$.

\begin{definition}[Policy]\label{def:policy}
We will call \emph{policy} any function $\Pi_{\mS}: \mathcal{R}_\mS \rightarrow [0,\pi]^M$ such that $\Phi^* = \Pi_{\mS}(\aT,\bT)$, with $\Phi^*$ being the optimal switching angles, solutions to Problem \ref{pb:SHEp} with target $(\aT,\bT)$.
\end{definition} 

With the aim of reconstructing the policy $\Pi_{\mS}$, a typical approach is to solve numerically Problem \ref{pb:SHE opt} for a limited number of points in $\mathbb{R}^{N_a+N_b}$ and check that the optimal value is zero. Secondly, one interpolates the function $\Pi_{\mS}$ in the convex set generated by the points previously obtained. Nevertheless, this approach has several difficulties and drawbacks.
\begin{itemize}
	\item[1.] \textit{Combinatory problem}: in practice, one does not dispose of a suitable waveform $\mS$ which yields a solution to the Problem \ref{pb:SHEp}. A common approach to solve the SHM problem consists in fixing the number of switches $M$, and then solve Problem \ref{pb:SHE opt} for  all the possible combinations of $M$ elements of $\mathcal{U}$. However, taking into account that the number of possible $M$-tuples  in $\mathcal U$ is of the order $(L-1)^M$, it is evident that the complexity of the above approach increases rapidly when $L>1$. This problem has been studied for instance in \cite{Yang2015,Yang2017} where, through appropriate algebraic transformations, the authors convert the SHM problem into a polynomial system whose solutions' set contains all the possible waveforms $\mathcal S$ of $M$ elements in $\mathcal{U}$. As a drawback of this approach, the number of switches $M$ needs to be prefixed. However, in some cases,  determining the number of switches which are necessary to reach the desired Fourier coefficients is not a straightforward task.
	
	\item[2.] \textit{Solvable set problem}: given a waveform $\mS$, the corresponding solvable set $\mathcal{R}_{\mS}$ is usually very small, yielding to policies $\Pi_{\mS}$ which are not very effective. This issue is typically addressed by solving Problem \ref{pb:SHE opt} for a set of waveforms $\{\mS_l\}_{l=1}^r$ and obtaining different policies $\{\Pi_{\mS_l}\}_{l=1}^r$ and solvable sets $\{\mathcal{R}_{\mS_l}\}_{l=1}^r$ for each one of them. By gathering them, one creates a new policy applicable in a wider range. However, since the solvable sets corresponding to different waveforms may be disjoint or even overlapping, this union of policies may give rise to regions where the solution for the same target $(\aT,\bT)$ is not unique, or even generate regions with no solution at all (see Fig. \ref{fig:chaos_policy}).
	
	\begin{SCfigure}[1.3][h]
		\centering
		\includegraphics[scale=0.3]{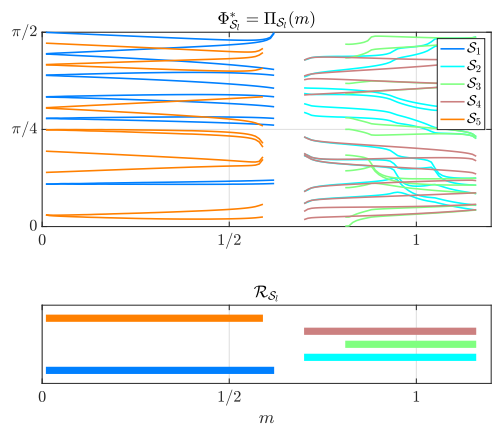}
		\caption{In the first picture, we display the optimal switching angles $\Phi^*_{\mS}$ associated to different waveforms $\{\mS_l\}_{l=1}^7$ for a SHM problem (see Remark \ref{remark:SHE}), considering $\mathcal{E}_a = \{1\}$ and $\mathcal{E}_b = \{1,5,7,11,13,17,19,23,25,29,31\}$. We chose $a_T = m$ for all $m\in [0,1.2]$ and $b_T = (0,\ldots,0)$. The second figure shows the solvable sets for each waveform we considered.}
		\label{fig:chaos_policy}
	\end{SCfigure}
	
	\item[3.] \textit{Policy problem}: due to the complexity of a policy generated by the union of different waveforms, the continuity of the switching angles cannot be guaranteed. This is a well known problem in the SHM community \cite{Agelidis2008,Dahidah2008,Dahidah2015,Yang2017} (see Fig. \ref{fig:chaos_policy}).	
\end{itemize}

As we shall see, all these mentioned criticalities may be overcome by our optimal control approach.

\section{SHM as an optimal control problem}\label{sec:Contributions}

Our main contribution in the present paper consists in formulating the SHM problem as an optimal control one. In this formulation, the Fourier coefficients of the signal $u(t)$ are identified with the terminal state of a controlled dynamical system of $N_a+N_b$ components defined in the time-interval $[0,\pi)$.  The control of the system is precisely the signal $u(t)$, defined as a function $[0,\pi)\to \mathcal{U}$, which has to steer the state from the origin to the desired values of the prescribed Fourier coefficients. The starting point of this approach is to rewrite the Fourier coefficients of the function $u(t)$ as the final state of a dynamical system controlled by $u(t)$. To this end, let us first note that, in view of \eqref{eq:an}, for all $u\in L^\infty ([0,\pi);\mathbb{R})$ any Fourier coefficient $a_j$ satisfies $a_j = y(\pi)$, with $y\in C([0,\pi);\mathbb{R})$ defined as
\begin{align*}
	y(t) = \dfrac{2}{\pi} \int_0^t u(\tau) \cos(j\tau) d\tau.
\end{align*}

Besides, thanks to the fundamental theorem of calculus, $y(\cdot)$ is the unique solution to the differential equation
\begin{equation}\label{eq: ODE Fourier}
	\begin{cases}
		\dot{y} (t) = \dfrac{2}{\pi} \cos(jt) u(t), \qquad  t\in [0,\pi)
		\\[5pt]
		y(0) = 0.
	\end{cases}
\end{equation}

Analogously, we can also write the Fourier coefficients $b_j$, defined in \eqref{eq:an}, as the solution at time $t=\pi$ of a differential equation similar to \eqref{eq: ODE Fourier}.

Hence, for $\mathcal{E}_a$, $\mathcal{E}_b$, $\aT$, and $\bT$ given, the SHM Problem \ref{pb:SHEp} can be reduced to finding a control function $u$ of the form \eqref{eq:uExpl}, satisfying \eqref{eq:staircase prop}, such that the corresponding solution ${\bf y} \in C([0,\pi);\mathbb{R}^{N_a+N_b})$ to the dynamical system
\begin{equation}\label{eq:forward dyn syst}
	\begin{cases}
		\dot{{\bf y}}(t) = \dfrac{2}{\pi} {\bf \mathcal{D}}(t) u(t), \qquad  t\in [0,\pi)
		\\[5pt]
		{\bf y}(0) = 0.
	\end{cases}
\end{equation}
satisfies ${\bf y} (\pi) = [\aT;\bT]^\top$, where
\begin{equation}\label{eq:Dynamics}
	{\bf \mathcal{D}}(t) = \left[{\bf \mathcal{D}}^a(t); {\bf \mathcal{D}}^b(t) \right]^\top, 
\end{equation}
with ${\bf \mathcal{D}}^a(t) \in \mathbb{R}^{N_a} $ and $ {\bf \mathcal{D}}^b(t) \in \mathbb{R}^{N_b}$ given by
\begin{align}\label{eq:DalphaDbeta}
	{\bf \mathcal{D}}^a(t) = \begin{bmatrix} \cos(e_a^1t) \\ \cos(e_a^2t) \\ \vdots \\ \cos(e_a^{N_a}t) \end{bmatrix},
	\quad {\bf \mathcal{D}}^b(t) = \begin{bmatrix} \sin(e_b^1t) \\ \sin(e_b^2t) \\ \vdots \\ \sin(e_b^{N_b}t) \end{bmatrix} 
\end{align}
Here, $e_a^i$ and $e_b^i$ denote the elements in $\mathcal{E}_a$ and $\mathcal{E}_b$, i.e.
\begin{align*}
	\mathcal{E}_a = \{e_a^1,e_a^2,e_a^3,\dots,e_a^{N_a}\}, \quad \mathcal{E}_b = \{e_b^1,e_b^2,e_b^3,\dots,e_b^{N_b}\}.
\end{align*}

In the sequel, and in order to simplify the notation, we reverse the time in \eqref{eq:forward dyn syst} using the transformation ${\bf x} (t) = {\bf y}(\pi - t)$. In this way, the SHM problem turns into the following null controllability one, for a dynamical system with initial condition ${\bf x}(0) = [\aT; \bT]^\top$ (see also Fig. \ref{fig:evolution_x}).

\begin{problem}[SHM via null controllability]\label{pb:SHEpControl}
Let $\mathcal{U}$ be given as in \eqref{eq:Udef}. Let $\mathcal{E}_a$, $\mathcal{E}_b$ and the targets $\aT$ and $\bT$ be given as in Problem \ref{pb:SHEp},  we look for a function $u: [0,\pi)\to [-1,1]$ of the form \eqref{eq:uExpl}, satisfying \eqref{eq:staircase prop}, such that the solution to the initial-value problem
\begin{equation}\label{eq:CauchyReversed}
	\begin{cases}
		\displaystyle\dot{{\bf x}}(t) = -\frac 2\pi{\bf \mathcal{D}}(t)u(t),  & t \in [0,\pi)
		\\[5pt]
		{\bf x}(0) = {\bf x}_0 := [\aT; \bT]
	\end{cases}
\end{equation}
satisfies ${\bf x} (\pi) = 0$, where ${\bf \mathcal{D}}$ is given by \eqref{eq:Dynamics}--\eqref{eq:DalphaDbeta}.
\end{problem}
\begin{SCfigure}[1][h] 
	\centering
	\includegraphics[scale=0.3]{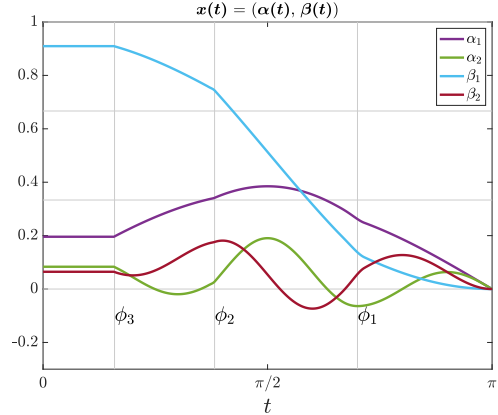}
	\caption{Evolution of the dynamical system \eqref{eq:CauchyReversed} with $\mathcal{E}_a = \{1,2\}$ and $\mathcal{E}_b = \{1,2\}$ corresponding to the control $u$ in Figure \ref{fig:exampleSHE}. The positions of the switching angles ${\bf \phi}$ are displayed as well.}\label{fig:evolution_x}
\end{SCfigure}

A natural approach for null controllability problems such as Problem \ref{pb:SHEpControl} is to formulate them as an optimal control one, where the cost functional to be minimized is the euclidean distance between the final state ${\bf x}(\pi)$ and the origin. In what follows, for a given vector ${\bf v}\in\mathbb{R}^d$, we denote by $\|{\bf v}\|$ the euclidean norm $\|{\bf v}\|_{\mathbb{R}^d}$. Let us introduce the set of admissible controls.
\begin{align*}
	&\mathcal A:= \Big\{u:[0,\pi)\to [-1,1]\; \text{ measurable}\Big\}
	\\
	&\mathcal A_{ad}:= \Big\{u\in\mathcal A \text{ of the form } \eqref{eq:uExpl} \text{ satisfying } \eqref{eq:staircase prop}\Big\}
\end{align*}

\begin{problem}[OCP for SHM]\label{pb:OCP1}
Let $\mathcal{U}$ be a given set as in \eqref{eq:Udef}. Let $\mathcal{E}_a$, $\mathcal{E}_b$ and the targets $\aT$ and $\bT$ be given as in Problem \ref{pb:SHEp}. We look for an admissible control $u\in \mathcal{A}_{ad}$ solution to the following optimal control problem:
\begin{equation*}
	\min_{u \in \mathcal{A}_{ad}}\;\frac 12 \|{\bf x}(\pi)\|^2 \quad \text{subject to the dynamics \eqref{eq:CauchyReversed}}.
\end{equation*}
\end{problem} 

\begin{remark}
Note that the cost functional in Problem \ref{pb:OCP1} is quadratic and, therefore, the existence of at least one minimizer is ensured for any target $[\aT,\bT]^\top$. However, we point out that such a minimizer is a solution to the SHM problem if and only if the minimum is equal to zero. When it is not the case, we say that the target $[\aT,\bT]^\top$ is unreachable, and then the SHM problem \ref{pb:SHEp} (resp. Problem \ref{pb:SHEpControl}) has no solution. In this work, we will not discuss the reachable set for the control problem \ref{pb:SHEpControl}.
\end{remark}

A main feature of the SHM problem is that we are looking for signal functions $u$ of the form \eqref{eq:uExpl} satisfying \eqref{eq:staircase prop}. In principle, this can be directly added as a constraint in the set of admissible controls $\mathcal{A}_{ad}$ as we did in Problem \ref{pb:OCP1}. However, considering an optimization problem in a non-convex set is not quite desirable. Indeed, it is well-known that mathematical optimization, in general, is an NP-hard problem, whereas for the case of convex optimization, algorithms with a polynomial computational time are available, as for instance, the interior point method \cite{helmberg1996interior}, the projected gradient descent \cite{calamai1987projected} or the penalty method \cite{eremin1967penalty}. In order to bypass this difficulty, we propose a variant of Problem \ref{pb:OCP1}, adding a penalization term for the control to the cost functional, and removing the staircase constraint on the control. 

\begin{problem}[Penalized OCP for SHM]\label{pb:OCP_penalizado}
Fix $\varepsilon>0$ and a convex function $\mathcal{L}\in C([-1,1];\mathbb{R})$.  Let $\mathcal{E}_a$, $\mathcal{E}_b$ and the targets $\aT$ and $\bT$ be given as in Problem \ref{pb:SHEp}. We look for a control $u\in \mathcal A$ solution to the following optimal control problem:
\begin{align*}
	&\displaystyle\min_{u \in \mathcal A}\;\left(\dfrac{1}{2} \|{\bf x}(\pi)\|^2 + \varepsilon \displaystyle\int_0^\pi \mathcal{L}(u(t)) dt\right) 
	\\[5pt] 
	&\text{subject to the dynamics \eqref{eq:CauchyReversed}}.
\end{align*}
\end{problem}

Observe that, in Problem \ref{pb:OCP_penalizado}, we do not impose the constraint that the control has to be of the form \eqref{eq:uExpl}, satisfying the staircase property \eqref{eq:staircase prop}. Nevertheless, as we shall see, these features of $u$ will arise naturally in the solution to Problem \ref{pb:OCP_penalizado}, from a suitable choice of the penalization term $\mathcal{L}$. 

Another important advantage of adding a penalization term for the control is that, as we shall prove in Theorems \ref{th:bang-bang} and \ref{th:PLP}, it ensures the uniqueness for the solution, and its the continuity with respect to the targets $\aT$ and $\bT$. 

On the contrary, one needs to take into account that the penalization term for the control might prevent the optimal trajectory from reaching the target. In other words, even if there exists a control for which the optimal trajectory satisfies ${\bf x}(\pi) = 0$, the optimal control in Problem \ref{pb:OCP_penalizado} might not do so, and therefore, the solution to Problem \ref{pb:OCP_penalizado} would not be a solution to the SHM problem. This issue may be controlled by a proper selection of the weighting parameter $\varepsilon$ which allows to tune the precision of the optimal control for the perturbed problem, guaranteeing that the final state of the optimal trajectory is close enough to zero. As a matter of fact, we can prove the following proposition.

\begin{proposition}\label{Prop:approx controllability}
Assume that $[\aT,\bT]^\top$ is such that Problem \ref{pb:SHEpControl} admits a solution, and let $u^\ast\in \mathcal A$ be the solution to Problem \ref{pb:OCP_penalizado}. Then the associated trajectory ${\bf x}^\ast\in C([0,\pi);\mathbb{R}^{N_a+N_b})$, solution to \eqref{eq:CauchyReversed}, satisfies
\begin{align*} 
	\|{\bf x}^\ast (\pi)  \|^2 \leq  4 \varepsilon \pi \| \mathcal{L}\|_\infty,
\end{align*}
where $\| \cdot\|_\infty$ denotes the max-norm in $C([-1,1]; \mathbb{R})$.
\end{proposition}

\begin{proof}
Since we are supposing that Problem \ref{pb:SHEpControl} has a solution, there exists a control $\tilde{u}\in \mathcal{A}_{ad}$ such that its corresponding trajectory $\tilde{\bf x}$, solution to \eqref{eq:CauchyReversed}, satisfies $\tilde{\bf x}(\pi) = 0$. 
	
Now, let $u^\ast\in \mathcal{A}$ be the solution to Problem \ref{pb:OCP_penalizado}, and let ${\bf x}^\ast$ be its corresponding trajectory. By the optimality of $u^\ast$ we have
\begin{align*}
	\frac{1}{2} \|{\bf x}^\ast(\pi)\|^2 +\varepsilon \int_0^\pi \mathcal{L}(u^\ast(\tau))d\tau \leq \varepsilon \int_0^\pi \mathcal{L}(\tilde{u}(\tau))d\tau,
\end{align*}
and hence, we deduce that $\|{\bf x}^\ast (\pi)\|^2 \leq 4 \varepsilon \pi \| \mathcal{L}\|_\infty.$
\end{proof}

Let us now describe the construction of penalization functions $\mathcal{L}$ which guarantee that any solution to Problem \ref{pb:OCP_penalizado} has the form \eqref{eq:uExpl} and satisfies \eqref{eq:staircase prop}. To this end, we will distinguish two cases, depending on the cardinality of $\mathcal{U}$.

\subsection{Bilevel SHM via OCP (Bang-Bang Control)} 

In this case, the control set $\mathcal{U}$ defined in \eqref{eq:Udef} has only two elements, i.e.  $\mathcal{U}=\{-1,1\}$. In the control theory literature, a control taking only two values is known as \emph{bang-bang control}. In the SHM literature, this kind of solution are called \textit{bi-level solutions}. Note that in this case, any $u$ with the form \eqref{eq:uExpl}  trivially satisfies the staircase property \eqref{eq:staircase prop}.

\begin{theorem}\label{th:bang-bang}
Let $\mathcal{U}=\{ -1, 1\}$, and ${\bf x}_0$ be given. For some $\alpha\in \mathbb{R}$ with $\alpha\neq 0$, consider the Problem \ref{pb:OCP_penalizado} with $\mathcal{L} (u) = \alpha\, u$. Then, the optimal control $u^\ast$, solution to Problem \ref{pb:OCP_penalizado} is unique and has a bang-bang structure, i.e. it is of the form \eqref{eq:uExpl}. In addition to that, the solution $u^\ast$ to Problem \ref{pb:OCP_penalizado} is continuous  with respect to ${\bf x}_0$ in the strong topology of $L^1(0,\pi)$.
\end{theorem}

The proof of Theorem \ref{th:bang-bang} is postponed to Section \ref{sec:Proof}, and follows from the optimality conditions given by the Pontryagin's maximum principle. In particular, the linearity of $\mathcal{L}$ and of the dynamical system \eqref{eq:CauchyReversed}, implies that the associated Hamiltonian is also linear, and then, it always attains its minimum at the limits of the interval $[-1,1]$.

We point out that, by choosing different penalizations $\mathcal{L}$, we may obtain solutions to the SHM problem with different waveforms due to the change of the Hamiltonian. See for instance Fig. \ref{fig:Bang-Bang-penalization}, where we have chosen $\mathcal{L}(u) = \pm u$.

\subsection{Multilevel SHM problem via OCP}

Inspired by the ideas of the previous subsection, we can address the case when $\mathcal{U}$ contains more than two elements. This is known in the power electronics literature as the \textit{multilevel SHM problem}. Now, the goal is to construct a function $\mathcal{L}$ such that the Hamiltonian associated to Problem \ref{pb:OCP_penalizado} always attains the minimum at points in $\mathcal{U}$. A way to construct such a function $\mathcal{L}$ is to interpolate a parabola in $[-1,1]$ by affine functions, considering the elements in $\mathcal{U}$ as the interpolating points. Since between any two points in $\mathcal{U}$,  the function $\mathcal{L}$ is a straight line,  the Hamiltonian is a concave function in these intervals, and hence, the minimum is always attained at points in $\mathcal{U}$.

\begin{theorem}\label{th:PLP}
Let ${\bf x}_0$ be given, and let $\mathcal{U}$ be a given set as in \eqref{eq:Udef}. For any $\alpha>0$ and $\beta\in \mathbb{R}$, set the function
\begin{align}\label{eq:parabola}
	\mathcal{P}(u) = \alpha (u-\beta)^2.
\end{align}
Consider Problem \ref{pb:OCP_penalizado} with 
\begin{align}\label{eq:PLP}
	\mathcal{L}(u) = \begin{cases}
	\lambda_k(u) & \text{if }  u \in [u_k,u_{k+1}) \\ \mathcal{P}(1) & \text{if } u = u_{L} 
	\end{cases} \quad \text{for all } k \in \{1,\dots,L-1\}, 
\end{align}
where 
\begin{align}\label{eq:lambda k}
	\lambda_k(u):= \dfrac{ (u-u_k)\mathcal{P}(u_{k+1}) + (u_{k+1}- u) \mathcal{P}(u_k)}{u_{k+1} - u_k}.
\end{align}
Assume in addition that $\mathcal{L}$ has a unique minimum in $[-1,1]$. Then, the optimal control $u^\ast$, solution to Problem \ref{pb:OCP_penalizado}, is unique and has the form \eqref{eq:uExpl} satisfying \eqref{eq:staircase prop}. Moreover, the solution $u^\ast$ to Problem \ref{pb:OCP_penalizado} is continuous with respect to ${\bf x}_0$ in the strong topology of $L^1(0,\pi)$.
\end{theorem}

The assumption of $\mathcal{L}$ having a unique minimum in $[-1,1]$ is actually necessary to ensure the staircase form \eqref{eq:uExpl} for the solution. Not assuming this hypothesis would entail the possibility of having continuous solutions for specific targets. See Fig. \ref{fig:sim-multi-level-par} for an illustration of this pathology. Nevertheless, the assumption of $\mathcal{L}$ having a unique minimizer can be easily ensured by choosing, for instance, $\beta=\pm 1$.

\begin{remark}
For completeness, we shall mention that, in Theorem \ref{th:PLP}, $\mathcal L$ can actually have a more general form, still yielding to a staircase optimal control $u^\ast$. Indeed, as we shall see in Section \ref{sec:Proof}, the proof of Theorem \ref{th:PLP} does not use the fact that $\mathcal P$ is a parabola. If we replace it with any other strictly convex function, our result remains valid. The choice we made of defining $\mathcal P$ as in \eqref{eq:parabola} is motivated by the fact that, most often, in optimal control theory the penalization terms are chosen to be quadratic.
\end{remark}

\begin{remark}[\emph{Bang-off-bang control}]
We note that when  $\mathcal{U}= \{-1,0,1\}$, we can just use the $L^1$-norm of the control as penalization, i.e. $\mathcal{L}(u) = \vert u\vert$. This yields to the so-called \emph{bang-off-bang} controls, that are widely studied in the literature \cite{nagahara2013maximum,ikeda2016maximum}. By taking a different parabola $\mathcal{P}$, one can then obtain different bang-off-bang solutions to the SHM problem.
\end{remark}

We illustrate in Fig. \ref{fig:examples_penalizations} different examples of penalization functions $\mathcal{L}$ giving rise to multilevel solutions to the SHM problem. We point out that, by varying the values of $\alpha$ and $\beta$ in Theorem \ref{th:PLP}, we can obtain solutions with different waveforms.
\begin{SCfigure}[1.2][h]
	\centering
	\includegraphics[scale=0.35]{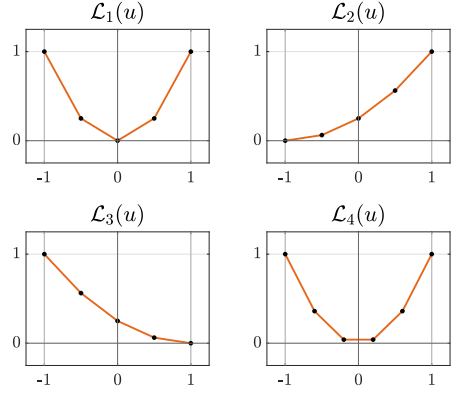}
	\caption{Some examples of convex piecewise affine penalization functions $\mathcal{L}$. The examples $ \mathcal {L}_1 $, $ \mathcal{L}_2$ and $ \mathcal{L}_3 $ satisfy the hypotheses of Theorem \ref{th:PLP}. On the contrary,  the function $\mathcal{L}_4 $ has not a unique minimizer, and then, we cannot ensure that the solution has staircase form.}
	\label{fig:examples_penalizations}
\end{SCfigure}

\section{Proofs of Theorems \ref{th:bang-bang} and \ref{th:PLP}}\label{sec:Proof}

We give here the proofs of the main results of this paper, i.e. Theorems \ref{th:bang-bang} and \ref{th:PLP}. 

At this regards, we notice that the existence of a minimizer, solution to Problem \ref{pb:OCP_penalizado}, can be easily proved employing the direct method in calculus of variations. Indeed, observe that the dynamical system \eqref{eq:CauchyReversed} is linear and the admissible controls in $\mathcal{A}$ are uniformly bounded. Moreover, the functional to be minimized is convex with respect to the control, which suffices to ensure its weak lower semicontinuity, allowing us to pass to the limit in the minimizing sequence.

For the sake of readability, we organize the rest of the proofs as follows: in subsection \ref{sec:opti cond}, we deduce the necessary optimality conditions from Pontryagin's Maximum Principle; in subsection \ref{sec: proof:bang-bang}, we prove that, when $\mathcal{L}$ is given as in Theorem \ref{th:bang-bang}, the solutions to Problem \ref{pb:OCP_penalizado} are bang-bang; in subsection \ref{sec: proof:PLP} we prove the analogous result for Theorem \ref{th:PLP}. Finally, in subsection \ref{sec: uniqueness continuity} we give the proof of uniqueness and continuity of the solution to Problem \ref{pb:OCP_penalizado} with respect to the initial condition, when the penalization term $\mathcal{L}$ is given as in Theorems \ref{th:bang-bang} or \ref{th:PLP}.

\subsection{Optimality conditions}\label{sec:opti cond}

The proofs of Theorems \ref{th:bang-bang} and \ref{th:PLP} are based on the optimality conditions for Problem \ref{pb:OCP_penalizado}, which can be deduced by means of Pontryagin's maximum principle \cite[Chapter~2.7]{bryson1975applied}. To this end, let us first introduce the Hamiltonian  associated to the Optimal Control Problem \ref{pb:OCP_penalizado}:
\begin{align}\label{eq:hamil}
	\mathcal{H}(t,{\bf p},u) = \varepsilon \mathcal{L}(u) - \frac 2\pi\big({\bf p} \cdot {\bf \mathcal{D}}(t)\big)u(t),
\end{align}
where ${\bf p}\in \mathbb{R}^{N_a+N_b}$ is the so-called adjoint variable, and arises from the restriction imposed by the dynamical system \eqref{eq:CauchyReversed}. In view of the definition of ${\bf \mathcal{D}}(t)$ in \eqref{eq:Dynamics}-\eqref{eq:DalphaDbeta}, we will sometimes write the state and the adjoint variables using the following notation:
\begin{align*}
	{\bf x}(t) = \begin{bmatrix} {\bf a}(t), {\bf b}(t) \end{bmatrix}^\top \quad \text{and}\quad
	{\bf p}(t) = \begin{bmatrix} {\bf p}^a(t), {\bf p}^b(t) \end{bmatrix}^\top.
\end{align*}

Now, let us derive the optimality conditions arising from Pontryagin's Maximum Principle.
\begin{itemize}
	\item[1.] \textbf{The adjoint system}: for any $u^\ast \in \mathcal{A}$ solution to Problem \ref{pb:OCP_penalizado}, there exists a unique adjoint trajectory ${\bf p}^\ast\in C([0,\pi); \mathbb{R}^{N_a+N_b})$ which satisfies the following terminal-value problem
	\begin{equation*}
		\begin{cases}
			\dot{\bf p}^\ast(t) = -\nabla_x \mathcal{H}(u(t),{\bf p}^\ast(t),t), \qquad t \in [0,\pi)
			\\[5pt]
			{\bf p}^\ast (\pi) = \nabla_x \Psi ({\bf x}^\ast (\pi))
		\end{cases}
	\end{equation*}
	where $\Psi({\bf x}) = \frac{1}{2} \|{\bf x}\|^2$ is the terminal cost. Moreover, since the Hamiltonian does not depend on the state variable ${\bf x}$, we simply have $\dot{{\bf p}^\ast}(t) = 0$ for all $t \in [0,\pi)$. We therefore deduce that the adjoint trajectory is constant, and given by
	\begin{equation}\label{eq:adjoint constant}
		{\bf p}^\ast (t) = {\bf x}^\ast (\pi), \quad \forall t \in [0,\pi). 
	\end{equation}
	
	\item[2.] \textbf{The Optimal  Control}: now, using the optimal adjoint trajectory, we can deduce the necessary optimality condition for the control, which reads as follows:
	\begin{align}\label{eq:control design}
		u^* (t) \in \argmin_{\vert u\vert\leq 1} \mathcal{H}(t,{\bf p}^*(t),u), \quad \forall t \in [0,\pi).
	\end{align}
	As we will see in subsections \ref{sec: proof:bang-bang} and \ref{sec: proof:PLP}, for functions $\mathcal{L}$ as the ones we consider in Theorems \ref{th:bang-bang} and \ref{th:PLP}, this argmin is a singleton for almost every $t\in [0,\pi)$. Hence, given the adjoint ${\bf p}^\ast$, the condition \eqref{eq:control design} uniquely determines the optimal control almost everywhere in $[0,\pi)$. The only points where the control is not uniquely determined are, precisely, the switching angles, i.e. the points of discontinuity of the solution.
\end{itemize}

In view of the form of the adjoint trajectory \eqref{eq:adjoint constant} associated to the optimal state trajectory ${\bf x}^\ast$, let us introduce the function
\begin{align}\label{eq:m ast}
	\mu^\ast (t) := \frac 2\pi \big({\bf x}^*(\pi) \cdot {\bf \mathcal{D}}(t)\big) = \sum_{j \in \mathcal{E}_a} a^*_j (\pi) \cos(jt) + \sum_{j \in \mathcal{E}_b} b^*_j (\pi) \sin(jt). 
\end{align}
Then, in view of \eqref{eq:hamil} and \eqref{eq:adjoint constant}, we can write the optimality condition \eqref{eq:control design} as
\begin{align}\label{eq:control design2}
	u^\ast(t)  \in & \argmin_{\vert u\vert\leq 1}  \mathcal{J} (u,\mu^\ast(t)).  
\end{align}    
where $\mathcal{J}$ is defined as
\begin{align}\label{eq:functionalJ}
	\mathcal{J} (u,\mu^\ast(t)):= \varepsilon \mathcal{L}(u) - \mu^\ast(t) u .
\end{align}    

We are now ready to prove that the solutions to Problem \ref{pb:OCP_penalizado}, when $\mathcal{L}$ is chosen as in Theorems \ref{th:bang-bang} and \ref{th:PLP}, have the desired staircase form \eqref{eq:uExpl}--\eqref{eq:staircase prop}.

\subsection{Proof of Theorem \ref{th:bang-bang} - Part 1}\label{sec: proof:bang-bang}

In this subsection we prove that, when $\mathcal{L}$ is given as in Theorem \ref{th:bang-bang}, the solutions to Problem \ref{pb:OCP_penalizado} are bang-bang.

\begin{proof}[Proof of Theorem \ref{th:bang-bang} (bang-bang structure of the control)]%

We need to prove that, if $\mathcal{L}(u) = \alpha u$ for some $0 \neq \alpha\in \mathbb{R}$, then any optimal control $u^\ast$ has the form \eqref{eq:uExpl} with $\mathcal{U}=\{-1,1\}$. Or in other words,  $u^\ast(t)$ takes values in $\mathcal{U}$ for all $t\in [0,\pi)$, except for a finite number of times.

Let $u^\ast\in \mathcal{A}$ be a solution to Problem \ref{pb:OCP_penalizado}, and let ${\bf x}^\ast$ be its associated optimal trajectory. We just need to notice that, due to \eqref{eq:control design2} and the choice of $\mathcal{L}$, $u^\ast$ satisfies
\begin{align*}
	u^\ast (t) = \begin{cases}
		-1 & \text{if} \   \mu^\ast(t) < \varepsilon\alpha 
		\\
		1 & \text{if} \  \mu^\ast(t) > \varepsilon\alpha
	\end{cases}.
\end{align*}

Observe that, when $\mu^\ast(t) = 0$, which corresponds only to the cases when ${\bf x}^\ast(\pi) = 0$, the optimal control is constant and is just given by  $u^\ast (t) =  -\text{sgn} (\alpha)$. In all the other cases, when ${\bf x}^\ast(\pi)\neq 0$, the function $\mu^\ast(t)$ is a linear combination of sines and cosines, and therefore, $\mu^\ast (t) = \varepsilon\alpha$ can only hold for a finite number of times $t\in [0,\pi)$, which are the discontinuity points of $u^\ast$ (the switching angles). Note that the choice of $u^\ast$ at these points is irrelevant as it represents a set of zero measure.  See Fig. \ref{fig:Bang-Bang-penalization} for a graphical illustration of the proof.
\end{proof}

\begin{SCfigure}[1.4][h] 
	\centering
	\includegraphics[scale=0.35]{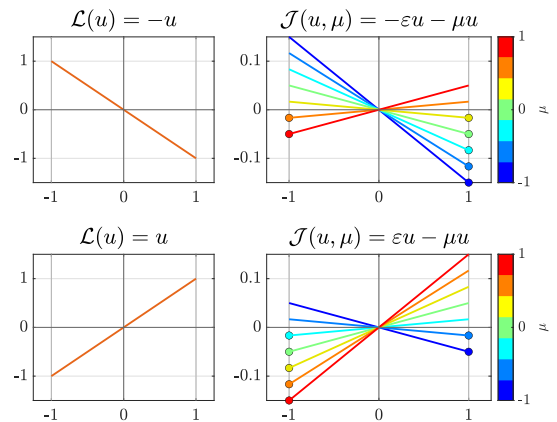}
	\caption{Bi-level SHE: in the left column, we see two examples of functions $\mathcal{L}$ as in Theorem \ref{th:bang-bang}. In the right column, we see the corresponding function $\mathcal{J}(\cdot,\mu)$ for different values of $\mu$. For each of them, the minumun is marked with a point.  We can see that the minimum is always attained at $-1$ or $1$.}\label{fig:Bang-Bang-penalization} 
\end{SCfigure}

\subsection{Proof of Theorem \ref{th:PLP} - Part 1}\label{sec: proof:PLP}
In this subsection we prove that, when $\mathcal{L}$ is given as in Theorem \ref{th:PLP}, the solutions to Problem \ref{pb:OCP_penalizado} have the multilevel structure. 

\begin{proof}[Proof of Theorem \ref{th:PLP} (multilevel control)]

In this case, we suppose that $\mathcal{U} = \{ u_k\}_{k=1}^L$ is a finite set of real numbers in $[-1,1]$ satisfying
\begin{align}\label{eq:order} 
	-1 = u_1 < u_2 <\ldots <u_L = 1, \quad \text{with} \ L> 2.
\end{align} 

The case $L=2$ is just the bi-level case. As in the previous proof, our goal is to show that the argmin in \eqref{eq:control design2} is a singleton and belongs to $\mathcal{U}$ for every $t\in [0,\pi)$ except for a finite number of points in $[0,\pi)$.

In this case, the study of the minimizers of $\mathcal{J}$ is slightly more involved since the penalization function $\mathcal{L}$ defined in \eqref{eq:PLP}-\eqref{eq:lambda k} is not differentiable at the points $u_k\in\mathcal U$. Since $\mathcal{L}$ is an affine interpolation of a convex function and, therefore, it is Lipschitz and convex, we deduce that also $\mathcal{J}$ is Lipschitz and convex as a function of $u$. In view of this, we have that $u^\ast$ minimizes $\mathcal{J} (u,\mu)$ if and only if
\begin{equation}\label{opti cond subdiff}
	0\in \partial_u \mathcal{J} (u^\ast,\mu),
\end{equation}
where $\partial_u$ denotes the subdifferential with respect to $u$. 

Let us recall below the definition of subdifferential from convex analysis:
\begin{align*}
	\partial_u \mathcal{J} (u,\mu) = \{c\in \mathbb{R} \quad \text{s.t.} \mathcal{J} (v,\mu) - \mathcal{J} (u,\mu) \geq c(v-u) \forall v\in [-1,1] \}. 
\end{align*}

For a convex function as $\mathcal{J}(\cdot, \mu)$, one can readily show that the subdifferential at $u\in (-1,1)$ is the nonempty interval $[a,b]$, where $a$ and $b$ are the one-sided derivatives
\begin{align*}
	a = \displaystyle\lim_{v\to u^-} \dfrac{\mathcal{J} (v,\mu) - \mathcal{J}(u,\mu)}{v-u}, \quad\quad\quad b = \displaystyle\lim_{v\to u^+} \dfrac{\mathcal{J} (v,\mu) - \mathcal{J}(u,\mu)}{v-u}. 
\end{align*}

Moreover, the subdifferential at $u=-1$ and $u=1$ is given by $(-\infty, b]$ and $[a,+\infty)$ respectively. Notice that, if $\mathcal J$ is differentiable at some $u\in (-1,1)$, then the left and the right derivatives coincide, and thus, $\partial_u \mathcal{J}(u,\mu)$ is just the classical derivative. Using this characterization of the subdifferential, we can compute $\partial_u\mathcal{J}(u,\mu)$ for all $u\in [-1,1]$ in terms of $\mu$. To this end, let us define
\begin{align*} 
p_k := \frac{d}{du}\lambda_k(u) = \frac{\mathcal{P}(u_{k+1}) - \mathcal{P} (u_k) }{u_{k+1} - u_k} 
\end{align*} 
for all $k\in \{1, \ldots, L-1\}$, with $\lambda_k(u)$ given by \eqref{eq:lambda k}. Using \eqref{eq:functionalJ} and \eqref{eq:PLP}, we can compute
\begin{align*}
	&\partial_u \mathcal{J} (-1,\mu) = (-\infty, \varepsilon p_1 -\mu], 
	\\[5pt]
	&\partial_u \mathcal{J} (1,\mu) = [\varepsilon p_{L-1} -\mu, +\infty), 
	\\[5pt]
	&\partial_u \mathcal{J} (u_k,\mu) = [\varepsilon p_{k-1} -\mu,  \, \varepsilon p_k -\mu],
\end{align*}
for all $k\in \{ 2, \ldots, L-1\}$, and
\begin{equation*}
	\partial_u \mathcal{J}(u,\mu) = \{\varepsilon p_k -\mu\},
\end{equation*}
for all $u\in (u_k, u_{k+1})$ and all $k\in \{ 1, \ldots, L-1 \}$. In view of the above computation, we obtain that
\begin{equation}\label{eq:subdiff}
	\begin{array}{ll}
		0\in \partial_u \mathcal{J} (-1,\mu) & \quad\text{iff}\quad  \mu\leq  \varepsilon p_1, 
		\\[5pt]
		0\in \partial_u \mathcal{J} (1,\mu) & \quad\text{iff} \quad \mu\geq  \varepsilon p_{L-1}, 
		\\[5pt]
		0\in \partial_u \mathcal{J} (u_k,\mu) & \quad\text{iff} \quad  \varepsilon p_{k-1} \leq \mu \leq \varepsilon p_k , 
	\end{array} 
\end{equation}
for all $k\in \{ 2, \ldots, L-1\}$,  and
\begin{equation}\label{eq:subdiff2}
	0\in \partial_u \mathcal{J} (u,\mu) \;\forall u\in [u_k, u_{k+1}] \qquad \text{iff} \ \mu= \varepsilon p_k
\end{equation}
for all $k\in \{ 1, \ldots, L-1 \}$.

Using \eqref{eq:subdiff}, along with the optimality condition \eqref{opti cond subdiff}, we deduce that, 
for a.e $\mu\in \mathbb{R}$, we have
\begin{equation}\label{eq:argmin singleton}
	\argmin_{\vert u\vert\leq 1} \mathcal{J} (u,\mu) = \{u_k\} \quad \text{for some} \ u_k\in \mathcal{U}.
\end{equation}
Indeed, \eqref{eq:argmin singleton} does not hold if and only if
\begin{equation}\label{eq:argmin non singleton}
	\mu = \varepsilon p_k \qquad \text{for some}\ k\in \{1,\ldots, L-1\}.
\end{equation}

Observe that, when $\mu^\ast(t) = 0$, which corresponds only to ${\bf x}^\ast(\pi) = 0$, the optimal control is constant and is just given by  $u^\ast (t) =  \argmin_{|u|\leq 1} \mathcal{L}(u)$ which, by hypothesis, is a singleton and belongs to $\mathcal{U}$ (note that between any two consecutive points of $\mathcal{U}$, the function $\mathcal{L}$ is a straight line). In all the other cases,  i.e. when ${\bf x}^\ast(\pi)\neq 0$, $\mu^\ast(t)$ is a linear combination of sines and cosines, and therefore, $\mu^\ast (t) = \varepsilon p_k$ can only hold, for each $k\in \{  1, \ldots, L-1 \}$, a finite number of times in $[0,\pi)$. These are precisely the discontinuity points of $u^\ast$ (the switching angles).

We have proved that, for all $t\in [0,\pi)$ except for a finite number of discontinuity points, which are precisely the switching angles $\{\phi_m\}_{m=0}^M$, we have $u^\ast(t)\equiv u_k$ for some $u_k\in\mathcal U$. Observe that, due to the continuity of $\mu^\ast(t)$, along with \eqref{eq:subdiff}, it is clear that $u^\ast(t)$ does not change value between two consecutive switching angles. Therefore, $u^\ast$ is piecewise constant, with a finite number of switches. The choice of $u^\ast$ at the discontinuity points is irrelevant as it represents a set of zero measure.

Finally, the staircase property \eqref{eq:staircase prop} can be deduced from \eqref{eq:control design2} and \eqref{eq:subdiff}, along with the continuity of the function $\mu^\ast(t)$. Following the same idea of Figure \ref{fig:Bang-Bang-penalization} for the Bang-Bang control, we can see  in Fig. \ref{fig:several_hamiltonian} a graphical interpretation of the proof for the multilevel case.
\end{proof} 

\begin{SCfigure}[0.9][h]
	\includegraphics[scale=0.35]{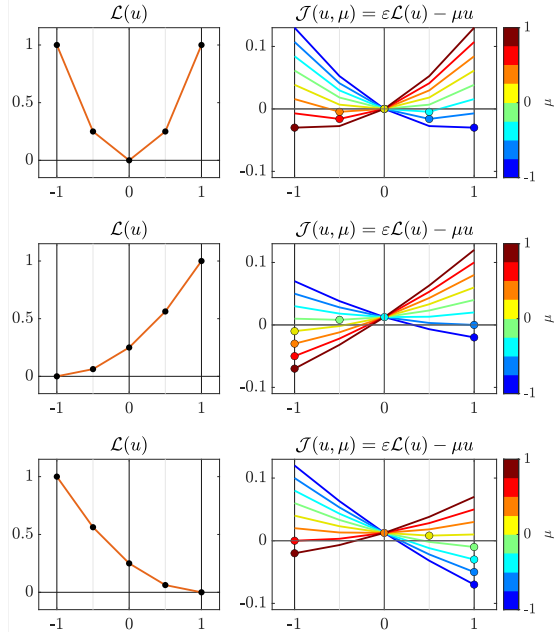}
	\caption{Multilevel SHM: in the left column, we see three different penalization functions $\mathcal{L}$ fulfilling the hypotheses of Theorem \ref{th:PLP}. In the right column, we see the corresponding function $\mathcal{J}(\cdot,\mu)$ for different values of $\mu$. For each of them, the minumun is marked with a point.  We can see that the minimum is always attained in $\mathcal{U}$.} 
	\label{fig:several_hamiltonian}
\end{SCfigure} 

\subsection{Uniqueness and continuity of solutions}\label{sec: uniqueness continuity}

The proofs in this subsection apply to both Theorems \ref{th:bang-bang} and \ref{th:PLP} (the bilevel and the multilevel case).

\medskip
\begin{proof}[Proof of Theorems \ref{th:bang-bang} and \ref{th:PLP} - uniqueness of solutions.]
We first prove that Problem \ref{pb:OCP_penalizado} admits a unique solution, i.e.  for each ${\bf x}_0\in \mathbb{R}^N$, there exists a unique $u^\ast\in \mathcal{A}$ minimizing the functional
\begin{equation}\label{eq:functional F}
	F(u,{\bf x}_0) := \dfrac{1}{2} \|{\bf x}(\pi) \|^2 + \varepsilon \int_0^\pi \mathcal{L}(u(t)) dt,
\end{equation}
where,  for each $u\in \mathcal{A}$, ${\bf x}(\pi)$ is given by
\begin{align*}
	{\bf x} (\pi) = {\bf x}_0  - \dfrac{2}{\pi} \int_0^\pi {\bf \mathcal{D}}(t) u(t) dt.
\end{align*}

We argue by contradiction. Suppose that there exist $u_1,u_2\in\mathcal{A}$ solutions to Problem \ref{pb:OCP_penalizado}, with $u_1\neq u_2$ in a set of positive measure. As both of them are optimal, using the arguments in subsections \ref{sec:opti cond}, \ref{sec: proof:bang-bang} and \ref{sec: proof:PLP}, we deduce that the controls $u_1$ and $u_2$ are uniquely determined a. e. in $[0,\pi)$ by the final state of the associated trajectory, i.e. ${\bf x}_1^\ast(\pi)$ and ${\bf x}_2^\ast(\pi)$, respectively. Therefore, if $u_1 \neq u_2$ in a set of positive measure, then we have ${\bf x}_1^\ast (\pi)\neq {\bf x}_2^\ast(\pi)$. Let us now consider the control
\begin{align*}
	\tilde{u} (t) = \dfrac{u_1(t) + u_2(t)}{2}.
\end{align*}
By the linearity of the dynamics \eqref{eq:CauchyReversed}, the convexity of $\mathcal{L}$, and using that ${\bf x}_1^\ast(\pi) \neq {\bf x}_2^\ast(\pi)$, we obtain
\begin{align*} 
	F(\tilde{u}, {\bf x}_0) &= \dfrac{1}{2} \left\| \dfrac{{\bf x}_1^\ast(\pi) + {\bf x}_2^\ast(\pi) }{2}\right\|^2 \!+ \varepsilon \int_0^\pi \mathcal{L} \left( \dfrac{u_1(t) + u_2(t) }{2}\right) dt < \dfrac{F(u_1,{\bf x}_0) + F(u_2, {\bf x}_0)}{2}. 
\end{align*}

Hence, using that both $u_1$ and $u_2$ minimize the functional $F(\cdot,{\bf x}_0)$, we obtain $F(\tilde{u},{\bf x}_0) < F(u_1,{\bf x}_0)$, which contradicts the optimality of $u_1$. We therefore conclude that the $u_1 (t) =u_2(t)$ for a.e. $t\in [0,\pi)$.
\end{proof}

\begin{proof}[Proof of Theorems \ref{th:bang-bang} and \ref{th:PLP} - continuity w.r.t. initial condition.]
Let us now give the proof of the $L^1$-continuity of the unique solution $u^\ast$ to Problem \ref{pb:OCP_penalizado} with respect to the initial condition. 
	
Let ${\bf x}_0$ be fixed. We need to prove that, for all $\gamma >0$, there exists $\delta >0$ such that 
\begin{align*}
	\|{\bf x}_1-{\bf x}_0\| \leq \delta \quad \text{implies} \quad \| u^\ast_1 - u^\ast_0\|_{L^1(0,\pi)} <\gamma,
\end{align*}
where $u^\ast_0$ and $u^\ast_1$ are the optimal controls corresponding to the initial conditions ${\bf x}_0$ and ${\bf x}_1$ respectively. 
	
As we have proved in subsections \ref{sec: proof:bang-bang} and \ref{sec: proof:PLP}, for any ${\bf x}_1$, the optimal control $u_1^\ast$, solution to Problem \ref{pb:OCP_penalizado}, is piecewise constant, taking values in $\mathcal{U}$, with a finite number of discontinuity points (switching points). Moreover, we claim that the number of switching points is bounded from above by a constant $M^\ast \in \mathbb{N}$, independent of  ${\bf x}_1\in \mathbb{R}^N$. Indeed, as we proved in subsection \ref{sec:opti cond}, the optimal control $u_1^\ast$ is determined by the optimality condition \eqref{eq:control design2}, using the function $\mu^\ast$ defined in \eqref{eq:m ast}. If $\mu^\ast \equiv 0$, then $u^\ast_1$ is constant and there are no switching points. In the other cases, $\mu^\ast(t)$ is a linear combination of sines and cosines with fixed frequencies.  In the bilevel case, in subsection \ref{sec: proof:bang-bang} we proved that the switching points correspond to the intersection points of $\mu^\ast(t)$ with $\varepsilon\alpha$. In the multilevel case, we proved in subsection \ref{sec: proof:PLP} that the switching points correspond to the intersections of $\mu^\ast(t)$ with $\varepsilon p_k$, see \eqref{eq:argmin non singleton}. In view of \eqref{eq:m ast}, as the frequencies are fixed, the number of these intersection points in the interval $[0,\pi)$ cannot exceed a certain number $M^\ast$, independent of the coefficients $a_j^\ast$ and $b_k^\ast$ in \eqref{eq:m ast}. Actually, $M^\ast$ only depends on  $\max\{\mathcal{E}_a, \mathcal{E}_b\}$ and the cardinality of $\mathcal U$. The claim then follows.
	
Using that, for any ${\bf x}_1$, the solution $u_1^\ast$ is piecewise constant taking values only in $\mathcal{U}$, and with a finite number of switches less than some $M^\ast$ independent of ${\bf x}_1$,  we deduce that there exists $K>0$, independent of ${\bf x}_1$ such that $\| u_1^\ast\|_{BV} \leq K$. See \eqref{BV norm} below for the definition of the $BV$ norm. We then obtain that, for any ${\bf x_1}\in \mathbb{R}^N$,
\begin{align*}
	u_1^\ast\in \mathcal{A}_K^\ast:= \{ u\in \mathcal{A}\, : \ \| u\|_{BV} \leq K \}.
\end{align*}

Now, for any $\gamma>0$ fixed, we can apply Lemma \ref{lem: compactness} below, to ensure the existence of $\eta>0$ such that
\begin{equation}\label{eq:low est boundary}
	F(u,{\bf x}_0)  \geq F(u_0^\ast,{\bf x}_0)  + \eta,
\end{equation}
for all $u\in \mathcal{A}_K^\ast$, with $\| u-u_0^\ast\|_{L^1(0,\pi)} =\gamma$. Since the set $\mathcal{A}_K^\ast$ is convex and $u_0^\ast$ minimizes $F(\cdot,{\bf x}_0)$, we can use \eqref{eq:low est boundary} and the convexity of the function $u\mapsto F(u,{\bf x}_0)$, to deduce that
\begin{equation}\label{eq:low est outside}
	F(u,{\bf x}_0)  \geq F(u_0^\ast,{\bf x}_0)  + \eta,
\end{equation}
for all $u\in \mathcal{A}_K^\ast$ such that  $\| u-u_0^\ast\|_{L^1(0,\pi)} \geq \gamma$.
	
Observe that, for any $u\in \mathcal{A}$, the function ${\bf x} \mapsto F(u,{\bf x})$  is  locally Lipschitz, and therefore, there exists a constant $C_F>0$ satisfying
\begin{equation}\label{eq:F Lipschitz}
	\vert F(u,{\bf x}_1)-F(u,{\bf x}_0)\vert\leq  C_F \|{\bf x}_1 - {\bf x}_0\|
\end{equation}
for any ${\bf x}_1$ such that $\|{\bf x}_1-{\bf x}_0\| \leq 1$. Notice that, since $u\in\mathcal{A}$ only takes values in $[-1,1]$, $C_F$ can be chosen independently of $u$.
	
Now,  combining  \eqref{eq:low est outside} and \eqref{eq:F Lipschitz}, for all $u\in \mathcal{A}_K^\ast$ such that $\| u-u_0^\ast\|_{L^1(0,\pi)} \geq \gamma$, we obtain
\begin{align}\label{estimate eta}
	F(u_0^\ast,{\bf x}_1) & \leq F(u_0^\ast,{\bf x}_0) +  C_F \|{\bf x}_1-{\bf x}_0\| \notag 
	\\
	& \leq F(u,{\bf x}_0) - \eta +  C_F \|{\bf x}_1- {\bf x}_0\| 
	\\
	& \leq F(u,{\bf x}_1) - \eta  + 2 C_F \|{\bf x}_1- {\bf x}_0\| \notag  
\end{align}

Finally, we can choose $\delta \in (0,1)$ such that $\delta <\frac{\eta}{4C_F}$,  and from \eqref{estimate eta}, we deduce that, if $\|{\bf x}_1 - {\bf x}_0\| \leq \delta$, then
\begin{align*} 
	F(u_0^\ast,{\bf x}_1)  \leq F(u,{\bf x}_1)  - \dfrac{\eta}{2}
\end{align*} 
for all $u\in \mathcal{A}_K^\ast$ such that  $\| u-u_0^\ast\|_{L^1(0,\pi)} \geq \gamma$, which then implies that necessarily $\| u_1^\ast-u_0^\ast\|_{L^1(0,\pi)} \leq \gamma$. This concludes the proof of the $L^1$-continuity of the solution with respect to the initial condition.
\end{proof}

Let us conclude the section with the following Lemma, which has been used in the previous proof.

\begin{lemma}\label{lem: compactness}
Let ${\bf x}_0\in \mathbb{R}^N$ be given and let $\mathcal{L}$ be as in Theorem \ref{th:bang-bang} or \ref{th:PLP}. Let $u^\ast_0\in\mathcal{A}$ be the unique solution to Problem \ref{pb:OCP_penalizado}. For any $K>0$, define the set of controls
\begin{equation}\label{A star K}
	\mathcal{A}_K^\ast := \{ u\in \mathcal{A}\, : \ \| u\|_{BV} \leq K \}.
\end{equation}
Then, for any $\gamma >0$, there exists $\eta:=\eta(\gamma ,K)>0$ such that $F(u^\ast_0,{\bf x}_0 ) \leq F(u,{\bf x}_0 ) - \eta$, for all $u\in \mathcal{A}_K^\ast$ such that $\|u - u_0^\ast\|_{L^1(0,\pi)} = \gamma$.
\end{lemma}

In the definition of $\mathcal{A}_K^\ast$,  we are considering measurable functions of bounded variation in $(0,\pi)$, i.e. functions whose distributional derivative is a Radon measure in $(0,\pi)$, that we denote by $|Du|$, and such that $|Du|(0,\pi)$ is finite.
We recall that the norm $\| \cdot\|_{BV}$ is defined as
\begin{equation}\label{BV norm}
	\| u\|_{BV} := \int_0^\pi \vert u(t)\vert dt + \vert Du\vert (0,\pi ).
\end{equation}

See \cite[Chapter 3]{ambrosio2000functions} for further details on the space of functions of bounded variation.

\begin{proof}[Proof of Lemma \ref{lem: compactness}]
We need to prove that $\mathcal{I}_{\gamma,K} = F(u_0^\ast,{\bf x}_0) + \eta$, for some $\eta>0$, where
\begin{align*}
	\mathcal{I}_{\gamma,K} :=  \inf \left\{ F(u,{\bf x}_0) \,: \, u\in \mathcal{A}_K^\ast,\, \|u- u_0^\ast \|_{L^1(0,\pi)} = \gamma \right\}.
\end{align*}

The result follows from the fact that the space $BV(0,\pi)$ is compactly embedded in $L^1 (0,\pi)$, see \cite[Theorem 3.23]{ambrosio2000functions}. Consider any minimizing sequence $u_n\in \mathcal{A}_K^\ast$ with $\| u_n - u^\ast_0\|_{L^1(0,\pi)} =\gamma$, satisfying
\begin{align*}
	\lim_{n\to+\infty} F(u_n,{\bf x}_0) = \mathcal{I}_{K,\gamma}.
\end{align*}

By \cite[Theorem 3.23]{ambrosio2000functions}, there exists a subsequence of $u_n$ which converges to some $\tilde{u}\in \mathcal{A}$, strongly in $L^1(0,\pi)$. From the continuity of the $L^1$-norm and of the functional $F(\cdot,{\bf x}_0)$ with respect to the strong $L^1$-topology, we deduce that the limit $\tilde{u}$ satisfies $\| \tilde{u} - u_0^\ast \|_{L^1(0,\pi)} = \gamma$ and $\quad F(\tilde{u},{\bf x}_0) = \mathcal{I}_{K,\gamma}$. Finally, since $u_0^\ast$ is the unique minimizer of $F(\cdot,{\bf x}_0)$, we conclude that
\begin{align*}
	\mathcal{I}_{\gamma, K} - F(u_0^\ast,{\bf x}_0) = F(\tilde{u},{\bf x}_0) - F(u_0^\ast,{\bf x}_0) = \eta >0. 
\end{align*}
\end{proof}

\section{Numerical simulations}\label{sec:Simulations}

In this section, we present several examples in which we implement the optimal control strategy we proposed to solve the SHM problem. All the simulations we are going to present can be found also in \cite{simus}. Our Experiments were conducted on a personal MacBook Pro laptop (1,4 GHz Quad-Core Intel Core i5, 8GB RAM, Intel Iris Plus Graphics 1536 MB). 

To solve our optimal control Problem \ref{pb:OCP_penalizado}, we will employ the direct method \cite{rao2009survey} which, in broad terms, consists in discretizing the cost functional and the dynamics, and then apply some optimization algorithm. The dynamics will be approximated with the Euler method, while for solving the discrete minimization problem we will employ the nonlinear constrained optimization tool \texttt{CasADi} \cite{Andersson2019}. \texttt{CasADi} is an open-source tool for nonlinear optimization and algorithmic differentiation which implements the interior point method via the optimization software \texttt{IPOPT} \cite{wachter2006implementation}. To be efficiently applied to solve an optimal control problem, we then need the functional we aim to minimize to be smooth. While this is clearly true in the bi-level case of Problem \ref{pb:OCP1}, the functional in Problem \ref{pb:OCP_penalizado}, due to the piecewise affine penalization, is not differentiable at the points $u_k\in\mathcal U$. For this reason, when treating the multilevel case, we will first need to build a smooth approximation of the  function $\mathcal L$ we introduced in \eqref{eq:PLP}. Once we have this approximation, we will employ the optimal control approach we presented in Section \ref{sec:Contributions} to solve some specific examples of SHM problem.

\subsection{Smooth approximation of $\mathcal L$ for multilevel control}

As we mentioned, to efficiently employ \texttt{CasADi} for solving our optimal control problem in the multilevel case, we need to build a smooth approximation of the cost functional. For this reason, we will regularize the function $\mathcal L$ defined in \eqref{eq:PLP} as follows. First of all, for all real parameter $\theta>0$, we define the $C^\infty(\mathbb{R})$ function
\begin{align*}
	\displaystyle h^\theta(x) := \frac{1 + \tanh(\theta x)}{2}.
\end{align*}
and observe that, for almost every $x\in \mathbb{R}$, $h^\theta(x)\to h(x)$ as $\theta\to +\infty$, where $h$ is the Heaviside function 
\begin{align*}
	h(x) = \begin{cases}
		1 & \text{ if } x > 0 
		\\
		0 & \text{ if } x \leq 0.
	\end{cases}
\end{align*}
Secondly, for all $k \in \{1,\dots,N_u-1\}$ we define the (smooth) function $\chi_{[u_k,u_{k+1})}^\theta:\mathbb{R} \rightarrow \mathbb{R}$ given by
\begin{align*}
	\chi_{[u_k,u_{k+1})}^\theta(x) := - 1 + h^\theta(x-u_k) + h^\theta(-x+u_{k+1}) = \frac{\tanh[\theta(x-u_k)] + \tanh[\theta (u_{k+1}-x)]}{2}
\end{align*}
which, as $\theta\to +\infty$, converges in $L^\infty(\mathbb{R})$ to the characteristic function $\chi_{[u_k,u_{k+1})}$. Finally, we define
\begin{align}\label{eq:Lsmooth}
	\mathcal{L}^\theta(u) = \sum_{k = 1}^{N_u-1} \lambda_k \chi^\theta_{[u_k,u_{k+1})}(u),
\end{align}
with $\lambda_k$ given by \eqref{eq:lambda k}, which, as $\theta\to +\infty$, converges in $L^\infty(\mathbb{R})$ to the penalization function $\mathcal L$ defined in \eqref{eq:PLP}.

Notice that this regularization is independent of the function $\lambda_k$ in \eqref{eq:PLP}, which is just required to be in the form \eqref{eq:lambda k}. Nevertheless, in our numerical experiments we shall select some specific $\lambda_k$. In particular, we will use 
\begin{gather}
	\lambda_k = (u_{k+1}+u_{k}) (u-u_k) + u_k^2, 
\end{gather}
which corresponds to taking $\alpha=1$ and $\beta=0$ in \eqref{eq:parabola}.

\subsection{Direct method for OCP-SHE}

To solve Problem \ref{pb:OCP_penalizado}, we use a direct method, whose starting point is to discretize the cost functional and the dynamics. To this end, let us consider a $N_t$-points partition of the interval $[0,\pi]$ 
\begin{displaymath} 
	\mathcal{T} = \{t_k\}_{k=1}^{N_t} 
\end{displaymath}
and denote by ${\bf u} \in \mathbb{R}^{N_t}$ the vector with components $u_k = u(t_k)$, $k=1,\ldots,N_t$. Then the optimal control problem (\ref{pb:OCP1}) can be written as optimization one with variable ${\bf u} \in \mathbb{R}^{N_t}$. In more detail, we can formulate the problem \ref{pb:OCP_penalizado} as the following one in discrete time.

\begin{problem}[Numerical OCP]\label{pb:numOCP2}
Given two sets of odd numbers $\mathcal{E}_a$ and $\mathcal{E}_b$ with cardinalities $\vert\mathcal{E}_a\vert = N_a$ and $\vert\mathcal{E}_b\vert = N_b$, respectively, the targets $\aT \in \mathbb{R}^{N_a}$ and $\bT \in \mathbb{R}^{N_b}$, and the partition $\mathcal{T}$ of $[0,\pi]$, we look for ${\bf u} \in \mathbb{R}^{N_t}$ that solves the following minimization problem:
\begin{align*}
	&\min_{{\bf u} \in \mathbb{R}^{N_t}} \Bigg[\|{\bf x}_{N_t}\|^2 + \varepsilon \sum_{k=1}^{N_t-1} 
	\bigg[\frac{\mathcal{L}^\theta(u_{t_k}) + \mathcal{L}^\theta(u_{t_{k+1}})}{2}\Delta t_k \bigg]  \Bigg]  
	\\[15pt]
	&\notag \text{subject to: } \begin{cases}
	{\bf x}_{t_{k+1}} = {\bf x}_{t_k} - \Delta t_{k} (2/\pi) {\bf \mathcal{D}}(t_k)\\
	{\bf x}_{t_1} = {\bf x}_0:= [\aT,\bT]^\top
	\end{cases} 
\end{align*}
where 
\begin{gather}
	\Delta t_{k} = t_{k+1} - t_{k}, \hspace{1em} \forall k \in \{1,\dots,N_t-1\}.
\end{gather}
\end{problem}

\subsection{Numerical experiments}

We now present several numerical experiments to show the effectiveness of our optimal control approach to solve SHM problems. All the examples share the following common parameters 
\begin{align*} 
	\varepsilon = 10^{-5},  \quad \theta = 10^5, \quad \text{and} \quad\mathcal{P}_t = \{0,0.1,0.2,\dots,\pi\}.
\end{align*} 
We consider the frequencies 
\begin{equation}\label{example ferquencies}
	\mathcal{E}_a = \mathcal{E}_b = \{1,5,7,11,13\},
\end{equation}
and the target vectors 
\begin{equation}\label{example targets}
	\aT = \bT = (m,0,0,0,0,0)^\top, \quad \forall m \in [-0.8,0.8].
\end{equation}

We shall consider three different control sets $\mathcal{U}$ which correspond to the aforementioned types of control:
\begin{itemize}
	\item[1.] Bang-bang control: $\mathcal{U} = \{-1,1\}$.
	\vspace{0.05cm}
	\item[2.] Bang-off-bang control:  $\mathcal{U} = \{-1,0,1\}$. 
	\vspace{0.05cm}
	\item[3.] 5-multilevel control: $\mathcal{U} = \{-1,-1/2,0,1/2,1\}$.
\end{itemize}

The results of our simulations are displayed in Fig. \ref{fig:sim-bang-bang}.
We have plotted the function
\begin{align*} 
	\begin{array}{cccc}
		\Phi: & [-0.8, 0.8]\times [0,\pi] & \longrightarrow & \mathcal{U} 
		\\
		& (m,t)  &\longmapsto & u_m^\ast (t),
	\end{array}
\end{align*}
where, for each $m \in [-0.8,0.8]$, $u_m^\ast (\cdot)$ represents the solution to the SHM problem with target frequencies as defined in \eqref{example ferquencies}-\eqref{example targets}. 

\begin{figure}[h]
	\begin{minipage}{0.32\textwidth}
		\centering
		\includegraphics[scale=0.35]{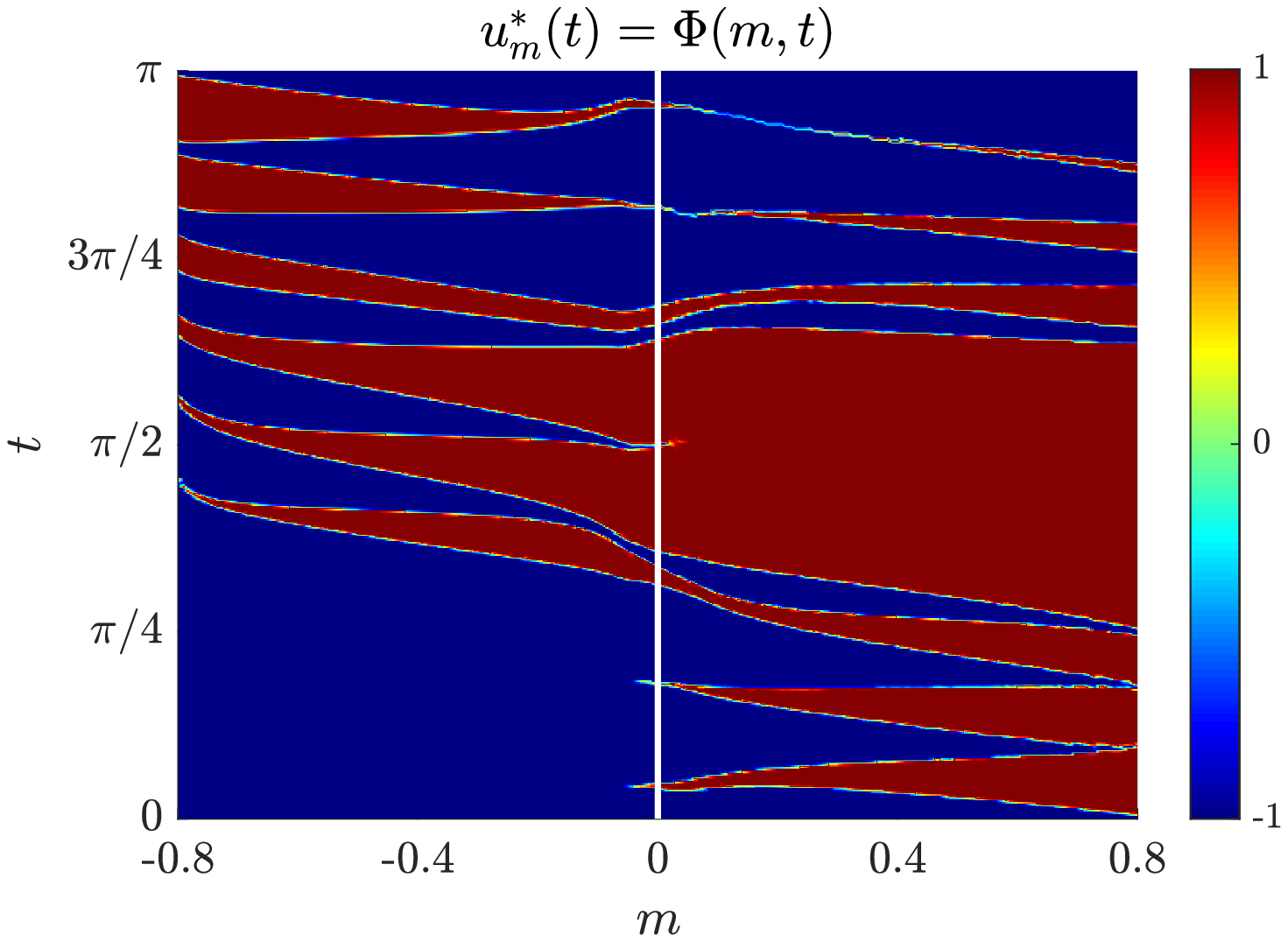}
	\end{minipage}
	\begin{minipage}{0.32\textwidth}
		\centering
		\includegraphics[scale=0.35]{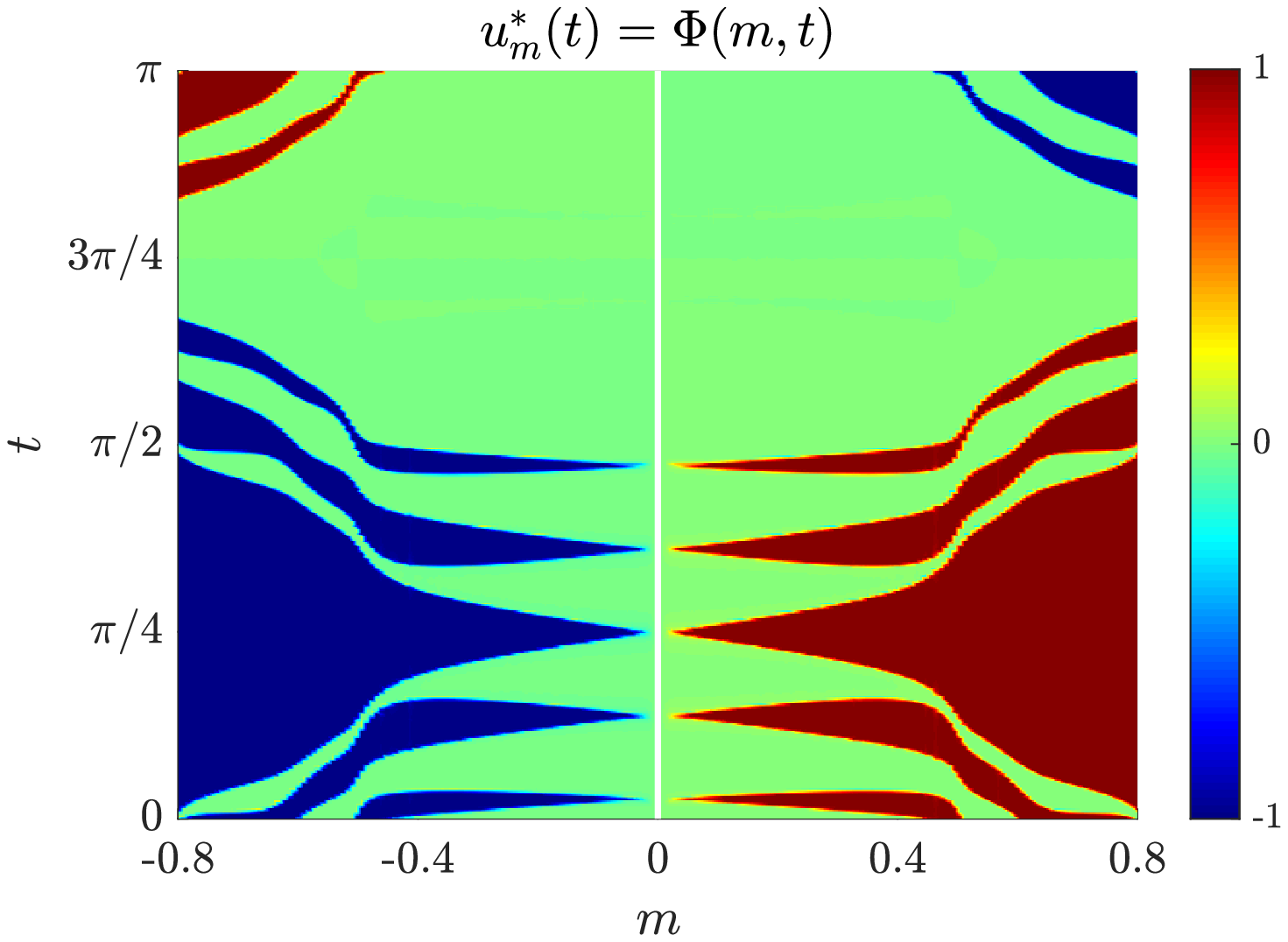}
	\end{minipage}
	\begin{minipage}{0.32\textwidth}
		\centering
		\includegraphics[scale=0.35]{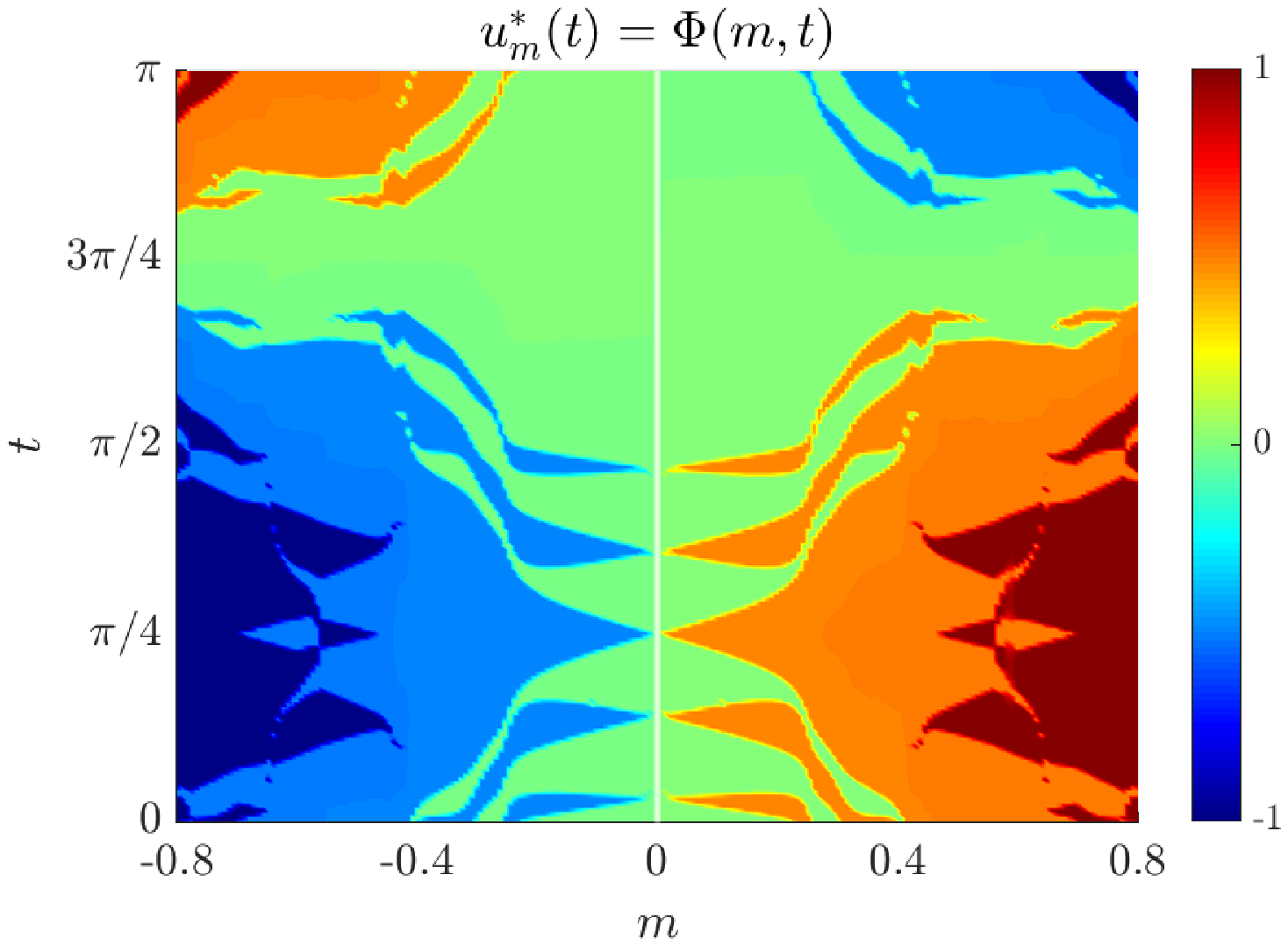}
	\end{minipage}
	\caption{Different types of control for the SHM problem: bang-bang (left), bang-off-bang (middle) and multilevel (right).}\label{fig:sim-bang-bang}
\end{figure} 

In Fig. \ref{fig:sim-bang-bang}, for each value of the parameter $m$ in the horizontal axis,  we observe that the optimal control, solution to Problem \ref{pb:OCP_penalizado} has the staircase structure introduced in Definition \ref{def:staircase prop}. The controls take values only in $\mathcal U$, which are represented by the different colors displayed at the right. For instance, in Fig. \ref{fig:sim-bang-bang}, the control is $u=-1$ in the blue region and $u=1$ in the red one. Note that,  the numerical results are in accordance with Theorems \ref{th:bang-bang} and \ref{th:PLP}. In addition to that, if we compare the policies $\Phi (m,\cdot) = u_m^\ast$ displayed in Fig. \ref{fig:sim-bang-bang} with the policies $\Pi_{\mS}$ of Fig. \ref{fig:chaos_policy}, we can see that the issues we mentioned in Section \ref{sec:SHE_finite-dim_pbm} concerning the solvable set and the continuity of the policy can be overcome by using our approach. In particular,  the optimal control formulation of SHM allows one to find solutions for an ample range of the parameter $m$, while considering always the same optimization Problem \ref{pb:numOCP2}. This is due to the fat that we are not restricting the solution to have a specific waveform. Furthermore, the \textit{combinatory problem} arising in the approach presented in Section \ref{sec:SHE_finite-dim_pbm} does not arise in our approach, as we do not need to launch an optimization process for all the possible waveforms for a given set $\mathcal U$. 

\vspace{0.5em}
\begin{remark}\label{counterexample}
Let us give an example which illustrates the necessity of assuming that the function $\mathcal{L}$ in Theorem \ref{th:PLP} has a unique minimizer in $[-1,1]$.
	
We consider the same parameters as in the above examples, but this time, the control set is given by
\begin{align*}
	\mathcal{U} = \{-1,-3/5,-1/5,1/5,3/5,1\}.
\end{align*}
	
This choice corresponds to the penalization function $\mathcal{L}_4$ represented in Fig. \ref{fig:examples_penalizations}. Observe that in this case
\begin{align*}
	\argmin_{\vert u\vert\leq 1} \mathcal{L}(u) = [-1/5, 1/5].
\end{align*}
	
In this case, the hypotheses of Theorem \ref{th:PLP} are not fulfilled and we cannot ensure that the solution has a staircase form. In Fig. \ref{fig:sim-multi-level-par},  we see that the solution is actually smooth for $m$ close to zero and takes values out of the control set $\mathcal{U}$. This stipulates that the assumption of $\mathcal{L}$ having a unique minimizer is necessary and cannot be removed if one wants to have a staircase solution. Notwithstanding, this issue can be overcome by choosing different values for the parameters $\alpha$ and $\beta$ in the definition of $\mathcal{L}$ in \eqref{eq:parabola}-\eqref{eq:lambda k}.
	
\begin{SCfigure}[1][h]
	\includegraphics[scale=0.4]{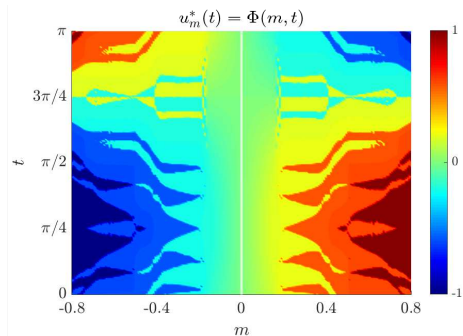}
	\caption{Solution to the optimal control problem \ref{pb:OCP_penalizado} with $\mathcal{L}$ not satisfying the uniqueness of the minimizer.}\label{fig:sim-multi-level-par}
\end{SCfigure} 
	
\end{remark}

\section{Conclusions}\label{sec:conclusions}

In this paper, we propose a novel optimal control based approach to the Selective Harmonic Modulation problem. More precisely, we have described how the SHM Problem \ref{pb:SHEp} can be reformulated in terms of a null-controllability one for which the solution $u$ plays the role of the control and can be obtained minimizing of a suitable cost functional. Besides, we have shown both theoretically and through numerical simulations that with our methodology we are able to solve several critical issues (described in detail in Section \ref{sec:SHE_finite-dim_pbm}) arising in practical power electronic engineering applications.
\begin{itemize}
	\item[1.] \textit{Combinatory problem}: in our approach, neither the waveform nor the number of switching angles need to be a priori determined, as they are implicitly established by the optimal control. This has two relevant advantages with respect to existing techniques as the one presented in Section \ref{sec:SHE_finite-dim_pbm}. On the one hand, this renders a computationally lighter methodology to solve the SHM problem, as it does not need to repeatedly solve an optimization problem for different waveforms. On the other hand, it bypasses the task of a priori estimating the number of switches which is necessary to reach the desired Fourier coefficients.
	
	\item[2.] \textit{Solvable set problem}: as we are not restricting the solution to have a prescribed waveform,  our approach provides solutions for an ample solvable set. 
	
	\item[3.] \textit{Policy problem}: the policy obtained through our methodology is not a gathering of several policies to which may correspond disjoint or even overlapping solvable sets. Hence, the continuity of the solution angles is guaranteed and we do not generate regions with no solution to the SHM problem.
\end{itemize}

However, some relevant issues are not completely covered by our study, and will be considered in future works:
\begin{itemize}
	\item[1.] \textbf{Minimal number of switching angles.} In practical applications, to optimize the converters' performance, it is required to maintain the number of switches in the SHM signal the lowest possible. It then becomes very relevant to determine which is the minimum number of switches allowing to reach the desired target Fourier coefficients.
	\item[2.] \textbf{Stability of the waveform and number of switching angles.} Related to the previous point, we observe in our numerical simulations in Section \ref{sec:Simulations} that, although the optimal control $u^\ast$ is $L^1$-continuous with respect to the initial condition,  the waveform and even the number of switching angles may change when varying the parameter $m$ continuously.  
	A finer analysis of the Problem \ref{pb:OCP_penalizado} may provide more information and understanding concerning this phenomenon. 
	\item[3.] \textbf{Characterization of the solvable set.} It would be interesting to have a full characterization of the solvable set for the SHM problem, thus determining the entire range of Fourier coefficients which can be reached by means of our approach.
	\item[4.] \textbf{Reduce the computational cost.} In this paper,  we have used existing numerical tools in optimal control to solve problem \ref{pb:OCP_penalizado}.
	It would be interesting to design algorithms adapted to our specific problem and compare their performance with other existing techniques in the SHM literature.
\end{itemize}


\begin{thebibliography}{9}
	
	\bibitem{Agelidis2008}
	{\sc Agelidis, V.~G., Balouktsis, A.~I. and Cossar, C.}
	\newblock {On attaining the multiple solutions of selective harmonic elimination PWM three-level waveforms through function minimization}.
	\newblock {\em IEEE Trans. Ind. Electron.}, 55.3 (2008), 996-1004.
	
	\bibitem{ambrosio2000functions}
	{\sc Ambrosio, L., Fusco, N. and Pallara, D.}
	\newblock {\em Functions of bounded variation and free discontinuity problems}.
	\newblock Courier Corporation, 2000.
	
	\bibitem{Andersson2019}
	{\sc Andersson, J. A.~E., Gillis, J., Horn, G., Rawlings, J.~B. and Diehl, M.}
	\newblock {CasADi} - A software framework for nonlinear optimization and optimal control.
	\newblock {\em Math. Program. Comput.}, 11.1 (2019), 1-36.
	
	\bibitem{bryson1975applied}
	{\sc Bryson, A.~E.}
	\newblock {\em Applied optimal control: optimization, estimation and control}.
	\newblock CRC Press, 1975.
	
	\bibitem{calamai1987projected}
	{\sc Calamai, P.~H. and Mor{\'e}, J.~J.}
	\newblock Projected gradient methods for linearly constrained problems.
	\newblock {\em Math. Programm.}, 39.1 (1987), 93-116.
	
	\bibitem{Dahidah2015}
	{\sc Dahidah, M.~S., Konstantinou, G. and Agelidis, V.~G.}
	\newblock {A review of Multilevel Selective Harmonic Elimination PWM: formulations, solving algorithms, implementation and applications}.
	\newblock {\em IEEE Trans. Power Electron.}, 30.8 (2015), 4091-4106.
	
	\bibitem{Dahidah2008}
	{\sc Dahidah, M. S.~A. and Agelidis, V.~G.}
	\newblock Selective harmonic elimination {PWM} control for cascaded multilevel voltage source converters: a generalized formula.
	\newblock {\em IEEE Trans. Power Electron.}, 23.4 (2008), 1620-1630.
	
	\bibitem{eremin1967penalty}
	{\sc Eremin, I.}
	\newblock The penalty method in convex programming.
	\newblock {\em Cybernetics}, 3.4 (1967), 53-56.
	
	\bibitem{helmberg1996interior}
	{\sc Helmberg, C., Rendl, F., Vanderbei, R.~J. and Wolkowicz, H.}
	\newblock An interior-point method for semidefinite programming.
	\newblock {\em SIAM J. Optim.}, 6.2 (1996), 342-361.
	
	\bibitem{ikeda2016maximum}
	{\sc Ikeda, T. and Nagahara, M.}
	\newblock Maximum hands-off control without normality assumption.
	\newblock {\em Proc. American Control Conference (ACC)} 
	(2016), 209-214.
	
	\bibitem{Konstantinou2010}
	{\sc Konstantinou, G.~S. and Agelidis, V.~G.}
	\newblock Bipolar switching waveform: novel solution sets to the selective harmonic elimination problem.
	\newblock {\em Proc. IEEE International Conference on Industrial Technology}
	(2010), 696-701.
	
	\bibitem{lee1999control}
	{\sc Lee, H.~W.~J., Teo, K.~L., Rehbock, V. and Jennings, L.~S.}
	\newblock Control parametrization enhancing technique for optimal discrete-valued control problems.
	\newblock {\em Automatica}
	35.8 (1999), 1401-1407.
	
	\bibitem{nagahara2013maximum}
	{\sc Nagahara, M., Quevedo, D.~E. and Ne{\v{s}}i{\'c}, D.}
	\newblock Maximum hands-off control and ${L}^1$ optimality.
	\newblock {\em 52nd IEEE Conference on Decision and Control} 
	(2013), 3825-3830.
	
	\bibitem{simus}
	{\sc Oroya, J.}
	\newblock djoroya/she-optimal-control-paper, github repository.
	\newblock \url{https://github.com/djoroya/SHE-Optimal-Control-paper}, 2021.
	\newblock Accessed: 2021-02-22.
	
	\bibitem{perez20172n}
	{\sc Perez-Basante, A., Ceballos, S., Konstantinou, G., Pou, J., Andreu, J. and de~Alegr{\'\i}a, I.~M.}
	\newblock (2n+1) selective harmonic elimination-PWM for modular multilevel converters: A generalized formulation and a circulating current control method.
	\newblock {\em IEEE Trans. Power Electron.}, 33.1 (2017), 802-818.
	
	\bibitem{rao2009survey}
	{\sc Rao, A.~V.}
	\newblock A survey of numerical methods for optimal control.
	\newblock {\em Adv. Astronaut. Sci.}, 135.1 (2009), 497-528.
	
	\bibitem{Sun1996}
	{\sc Sun, J., Beineke, S. and Grotstollen, H.}
	\newblock Optimal PWM based on real-time solution of harmonic elimination equations.
	\newblock {\em IEEE Trans. Power Electron.}, 11.4 (1996), 612-621.
	
	\bibitem{Sun1992}
	{\sc Sun, J. and Grotstollen, H.}
	\newblock Solving nonlinear equations for selective harmonic eliminated PWM using predicted initial values.
	\newblock {\em Proc. International Conference on Industrial Electronics, Control, Instrumentation, and Automation - Vol. 1} 
	(1992), 259-264.
	
	\bibitem{wachter2006implementation}
	{\sc W{\"a}chter, A. and Biegler, L.~T.}
	\newblock On the implementation of an interior-point filter line-search algorithm for large-scale nonlinear programming.
	\newblock {\em Math. Programm.}, 106.1 (2006), 25-57.
	
	\bibitem{wu2009filled}
	{\sc Wu, C.~Z., Teo, K.~L. and Rehbock, V.}
	\newblock A filled function method for optimal discrete-valued control problems.
	\newblock {\em J. Global Optim.}, 44.2 (2009), 213--225.
	
	\bibitem{Yang2015}
	{\sc Yang, K., Yuan, Z., Yuan, R., Yu, W., Yuan, J. and Wang, J.}
	\newblock A Groebner bases theory-based method for selective harmonic elimination.
	\newblock {\em IEEE Trans. Power Electron.}, 30.12 (2015), 6581-6592.
	
	\bibitem{Yang2017}
	{\sc Yang, K., Zhang, Q., Zhang, J., Yuan, R., Guan, Q., Yu, W. and Wang, J.}
	\newblock Unified selective harmonic elimination for multilevel converters.
	\newblock {\em IEEE Trans. Power Electron.}, 32.2 (2017), 1579-1590.
	
	\bibitem{yu2013optimal}
	{\sc Yu, C, Li, B., Loxton, R. and Teo, K.~L.}
	\newblock Optimal discrete-valued control computation.
	\newblock {\em J. Global Optim.}, 56.2 (2013), 503-518.
\end{thebibliography}
\end{document}